\documentclass[a4paper,12pt, reqno]{amsart}

\usepackage{amsfonts}
\usepackage{amsmath}
\usepackage{amssymb}
\usepackage{mathrsfs}
\usepackage{hyperref}
\usepackage{graphicx}
\usepackage{comment}
\usepackage[usenames]{color}

\numberwithin{equation}{section}
\theoremstyle{plain}
\newtheorem{thm}{Theorem}[section]

\newtheorem{cor}[thm]{Corollary}
\newtheorem{lemma}[thm]{Lemma}

\theoremstyle{definition}

\newtheorem{rem}[thm]{Remark}

\begin{document}

\title[Cylindrical critical Sobolev type inequality and identities]
{Cylindrical extensions of critical Sobolev type inequalities and identities
}

\author[M. Ruzhansky]{Michael Ruzhansky}
\address{
  Michael Ruzhansky:
 \endgraf
    Department of Mathematics: Analysis, Logic and Discrete Mathematics
  \endgraf
    Ghent University, Belgium
   \endgraf
  and
  \endgraf
  School of Mathematical Sciences
    \endgraf
    Queen Mary University of London, United Kingdom
   \endgraf
  {\it E-mail address} {\rm
michael.ruzhansky@ugent.be}
  }
\author[Y. Shaimerdenov]{Yerkin Shaimerdenov}
\address{
  Yerkin Shaimerdenov:
    \endgraf
  SDU University
Kaskelen, Kazakhstan
\endgraf
  and
  \endgraf
    Department of Mathematics: Analysis, Logic and Discrete Mathematics
    \endgraf
    Ghent University, Belgium
    \endgraf
    and
    \endgraf
    Institute of Mathematics and Mathematical Modeling, Kazakhstan
    \endgraf
  {\it E-mail address} {\rm yerkin.shaimerdenov@sdu.edu.kz}
  }
\author[N. Yessirkegenov]{Nurgissa Yessirkegenov}
\address{
  Nurgissa Yessirkegenov:
  \endgraf
    KIMEP University, Almaty, Kazakhstan
  \endgraf
  {\it E-mail address} {\rm nurgissa.yessirkegenov@gmail.com}
  }

\thanks{This research is funded by the Committee of Science of the Ministry of Science and Higher Education of the Republic of Kazakhstan (Grant No. AP23490970), by the FWO Odysseus 1 grant G.0H94.18N:
Analysis and Partial Differential Equations, and by the Methusalem programme of the Ghent University Special Research Fund (BOF) (Grant number 01M01021). MR is also supported by EPSRC
grant EP/V005529/1 and FWO grant G022821N}


\begin{abstract}
In this paper, we investigate cylindrical extensions of critical Sobolev type (improved Hardy) inequalities and identities in the style of Badiale-Tarantello \cite{BT02}, which in a special case give a critical Hardy inequality and its stability results. We also obtain higher-order identities, which interestingly include well-known numbers like double factorial, Oblong numbers, and Stirling numbers of the second kind. All functional identities are obtained in $L^{p}$ for $p\in (1,\infty)$ without the real-valued function assumption, which gives a simple and direct understanding of the corresponding inequalities as well as the nonexistence
of nontrivial extremizers. As applications, we obtain Caffarelli-Kohn-Nirenberg type inequalities with logarithmic weights, which in a particular case give the critical case of the Heisenberg-Pauli-Weyl type uncertainty principle. We also discuss these results in the setting of Folland and Stein's homogeneous Lie groups. A special focus is devoted to stratified Lie groups, where Sobolev type inequalities become intricately intertwined with the properties of sub-Laplacians and more general subelliptic partial differential equations. The obtained results are already new even in the classical Euclidean setting with respect to the range of parameters and the arbitrariness of the choice of any homogeneous quasi-norm. Most inequalities are obtained with sharp constants. 
\end{abstract}

\maketitle

\section{Introduction}
Let us start by recalling the classical Sobolev inequality. For a function $f \in W^{1,q}(\mathbb{R}^n)$ with $1<q<n$ there holds
\begin{equation} \label{clas.sob}
\|f\|_{L^{p}(\mathbb{R}^n)} \le C \|\nabla f\|_{L^q(\mathbb{R}^n)}, 
\end{equation}
where $p = \frac{nq}{n - q},$ and $C$ is a positive constant that depends only on $n$ and $q$.

The following Sobolev type inequality was obtained in \cite{OS09} and \cite{BEHL08} by replacing the standard gradient in \eqref{clas.sob} with the Euler operator $x\cdot\nabla$:
\begin{equation} \label{sob.type}
\|f\|_{L^p(\mathbb{R}^n)} \leq \frac{p}{n} \|x \cdot \nabla f\|_{L^q(\mathbb{R}^n)}.
\end{equation}
For any $\lambda > 0$, by substituting $f(x) = h(\lambda x)$ into \eqref{sob.type}, one easily observes that $p = q$ is a necessary condition to have \eqref{sob.type}.

Now, let us recall the classical Hardy inequality.
For any function $f\in C^\infty_0(\mathbb{R}^n)$ and for $1<p<n,$ we have
\begin{equation} \label{har}
\left\|\frac{f}{|x|}\right\|_{L^p(\mathbb{R}^n)} \leq \frac{p}{n-p} \|\nabla f\|_{L^p(\mathbb{R}^n)}, 
\end{equation}
where $\frac{p}{n-p}$ is the sharp constant and $|x|$ is the Euclidean norm on $\mathbb{R}^n$.

In \cite{BT02}, the following extension of \eqref{har} was obtained:
Let $x=(x',x'') \in\mathbb{R}^N \times \mathbb{R}^{n-N}$, $2\leq N \leq n$. For any function $f\in C^\infty_0(\mathbb{R}^n)$ and for $1<p<N,$ we have
\begin{equation} \label{ext.har}
\left\|\frac{f}{|x'|}\right\|_{L^p(\mathbb{R}^n)} \leq \frac{p}{N-p} \|\nabla f\|_{L^p(\mathbb{R}^n)}, 
\end{equation}
where $|x'|$ is the Euclidean norm on $\mathbb{R}^N$ and $\frac{p}{N-p}$ was conjectured to be the sharp constant.
Sharpness of the constant was first proved in \cite{SSW03} and then in \cite{RS17b} by using alternative ways. The reason for this kind of extension was its application to a nonlinear PDE which was proposed as a model of the dynamics of galaxies by Bertin \cite{Ber00} and Ciotti \cite{Cio01}.

In \cite[Sect. 2.1.7, Corollary 3]{Maz11} Maz'ya  established the extended Hardy inequality \eqref{ext.har} with a remainder term when $p=2$: Let $n>2,2<q<\frac{2n}{n-2}$, and $\gamma=-1+n\left(\frac{1}{2}-\frac{1}{q}\right)$. Then there exists a positive constant $c$ such that
\begin{equation}\label{Maz_ineq}
\left(\frac{N-2}{2}\right)^{2}\left\|\frac{f}{|x'|}\right\|_{L^{2}(\mathbb{R}^{n})}^{2}\leq \|\nabla f\|_{L^{2}(\mathbb{R}^{n})}^{2}-c\left(\int_{\mathbb{R}^{n}}|u|^q|x'|^{\gamma q} \mathrm{~d} x\right)^{\frac{2}{q}}
\end{equation}
holds for any $f \in C_0^{\infty}\left(\mathbb{R}^{n}\right)$, subject to the condition $f(0, x'')=0$ in the case $N=1$. Note that the inequality \eqref{Maz_ineq} also implies the Sobolev inequality with a cylindrical weight when we drop the left-hand side.

Recently, in \cite{KY24} the following improved version of the inequality \eqref{sob.type} with a more general weight was obtained:
Let $x=(x',x'') \in \mathbb{R}^N\times\mathbb{R}^{n-N}, 1\leq N\leq n$ and $\alpha\in\mathbb{R}$. Then for any $f \in C^\infty_0(\mathbb{R}^n\backslash\{x'=0\})$, and all $1<p<\infty$, we have
\begin{equation} \label{sob.type with weights}
\frac{|N-\alpha p|}{p}\left\|\frac{f}{|x'|^\alpha}\right\|_{L^p(\mathbb{R}^n)}\leq\left\|\frac{x'\cdot\nabla_N f}{|x'|^\alpha}\right\|_{L^p(\mathbb{R}^n)},
\end{equation}
where $\nabla_N$ is the standard Euclidean gradient on $\mathbb{R}^N.$ If $\alpha\neq \frac{N}{p}$, then the constant $\frac{|N-\alpha p|}{p}$ is sharp. This result is also related to the inequalities of type \eqref{ext.har} and \eqref{Maz_ineq} in the sense that applying the Cauchy-Schwarz inequality to the right-hand side of \eqref{sob.type with weights} implies \eqref{ext.har} with a more general weight. Actually, the authors in \cite{KY24} obtained \eqref{sob.type with weights} with a sharp remainder. We also refer to \cite[Corollary 1.1 and 1.2]{DP21} for the $L^{2}$ case of \eqref{sob.type with weights} with a sharp remainder as well as to \cite{DN20} for cylindrical Hardy inequalities on half-spaces. 

Moreover, in the same work \cite{KY24}, the following Caffarelli–Kohn–Nirenberg type inequality was obtained:
Let $x=(x',x'') \in \mathbb{R}^N\times\mathbb{R}^{n-N}, 1\leq N\leq n$ and $\alpha\in\mathbb{R}$. Let $1 < p, q < \infty$, $0 < r < \infty$ with $p+q\geq r$ and $\delta \in [0, 1] \cap \left[\frac{r-q}{r},\frac{p}{r}\right]$ and $a,b,c \in \mathbb{R}$. Assume that $\frac{\delta r}{p} + \frac{(1-\delta)r}{q} = 1$ and $c = \delta(a-1) + b(1-\delta)$. Then we have the following Caffarelli-Kohn-Nirenberg type inequality for all $f \in C_0^\infty(\mathbb{R}^n\setminus\{{0}\})$:
\begin{equation} \label{CKN with weights}
\| |x'|^c f \|_{L^r(\mathbb{R}^n)}\leq \left|\frac{p}{N+p(a-1)}\right|^\delta \||x'|^a\nabla_N f\|^\delta_{L^p(\mathbb{R}^n)}\||x'|^b f\|^{1-\delta}_{L^q(\mathbb{R}^n)}.\end{equation}

Inequalities of \eqref{sob.type with weights} and \eqref{CKN with weights} type for the horizontal gradient with the Euclidean norm on the first stratum of a stratified Lie group were given in \cite{RSY17a}. We can refer to \cite{CCR15}, \cite{RY24b}, and \cite{RY24a} for more information on Hardy inequalities for hypoelliptic operators on stratified Lie groups, graded Lie groups, and general Lie groups, respectively.

In this paper, we investigate the critical case of inequalities \eqref{sob.type with weights} and \eqref{CKN with weights} when $N=\alpha p$ for the first inequality and $N=p(1-a)$ for the latter one. Moreover, we show their sharp remainder formulae as well as higher-order versions. Interestingly, the coefficients in higher-order identities include well-known numbers like double factorial, Oblong numbers, and Stirling numbers of the second kind. Higher order identities are obtained by iteration of the identity, including the Euler operator $x\cdot\nabla$. The relation between Stirling numbers of the second kind and the Euler operator was mentioned by H. Scherk in his Ph.D. thesis \cite{Sch1823} in 1823. He showed that the iteration of $x\frac{d}{dx}$ produces Stirling numbers of the second kind as coefficients, and one has 
\begin{equation} \label{Eul.op.Stirling}
\left(x\frac{d}{dx}\right)^m=\sum_{n=1}^{m}S(m,n)x^n
\left(\frac{d}{dx}\right)^n,
\end{equation}
where $S(m,n)$ are the Stirling numbers of second kind. To find more about relations between derivative operators and Stirling numbers, we refer to \cite{M10}, \cite{L00}, and the references therein.
Identity \eqref{Eul.op.Stirling} was studied in quantum
physics in analyzing creation and annihilation operators \cite{BF10},  \cite{BPS03}, \cite{MBP05}, and its combinatorial aspect was investigated in  \cite{EGH15} and
the references therein. Here, we show different relations of Stirling numbers of the second kind and the multidimensional Euler operator in the setting of $L^{2}$ space.  

We now summarise briefly the main results of this paper in the Euclidean setting: 
\begin{itemize}
    \item(\textbf{Critical $L^{p}$-Sobolev type inequality and identities}) 
    Let $x=(x',x'') \in \mathbb{R}^N\times\mathbb{R}^{n-N}, 1\leq N\leq n$, $|x'|$ is the Euclidean norm on $\mathbb{R}^N,$ $\nabla_N$ is the standard gradient on $\mathbb{R}^N$ and $1<p<\infty.$ 
    \begin{itemize}
        \item[(i)] Then for any complex-valued function $f \in C^\infty_0(\mathbb{R}^n\backslash\{{x'=0}\})$ we have 
        \begin{equation} \label{sob in intro}
        \left\|{\frac{f}{|x'|^\frac{N}{p} }}\right\|_{L^p(\mathbb{R}^n)}
        \leq p \left\|\frac{x'\cdot \nabla_N f}{|x'|^\frac{N}{p}}\log|x'|\right\|_{L^p(\mathbb{R}^n)},
        \end{equation}
        where $p$ is sharp.
        \item[(ii)]
        For any complex-valued function $f \in C^\infty_0(\mathbb{R}^n\backslash\{{x'=0}\})$  we have the following identity which is the sharp remainder formula for (\ref{sob in intro}):
\begin{multline} \label{identity in intro}
    \left\|\frac{f}{|x'|^\frac{N}{p}}\right\|_{L^p(\mathbb{R}^n)}^p=p^p\left\|\frac{\log|x'|(x'\cdot \nabla_N) f}{|x'|^\frac{N}{p}} \right\|_{L^p(\mathbb{R}^n)}^p-\\
    -\int_{\mathbb{R}^n}C_p\left(\frac{p\log|x'|(x'\cdot \nabla_N) f}{|x'|^\frac{N}{p}} ,\frac{p(\log|x'|)^{-\frac{1}{p}+1}}{|x'|^\frac{N}{p}}(x'\cdot\nabla_N)\left(f(\log|x'|)^\frac{1}{p}\right)\right)dx,
\end{multline}
where the functional $C_p$ is given by
\begin{equation}
C_p(u,v):=|u|^p-|u-v|^p-p|u-v|^{p-2}{\rm{Re}}((u-v)\cdot\overline{v})\geq0.\end{equation}
Moreover, for $p\geq2$, $C_p$ vanishes if and only if $$f=(\log|x'|)^{-\frac{1}{p}}\varphi\left(\frac{x'}{|x'|},x''\right)$$ 
for some function $\varphi:\mathbb{S}^{N-1}\times\mathbb{R}^{n-N}\rightarrow\mathbb{C}$, which makes the left-hand side of \eqref{identity in intro} infinite unless $\varphi=f=0$ on the basis of the non-integrability of the function
$$\frac{|f|^p}{|x'|^N}=|\log|x'||\frac{\left|\varphi\left(\frac{x'}{|x'|},x''\right)\right|^p}{|x'|^N}$$
on $\mathbb{R}^n$. Consequently, the sharp constant $p$ in \eqref{sob in intro} is not attained on $C^\infty_0(\mathbb{R}^n\backslash\{{x'=0}\}).$
    \end{itemize}  
    \item(\textbf{$L^{2}$-Higher order critical Sobolev type identity})
    Let $x=(x',x'') \in \mathbb{R}^N\times\mathbb{R}^{n-N}, 1\leq N\leq n,$ and $k\in \mathbb{N}.$  Then for any $f \in C^\infty_0(\mathbb{R}^n\backslash\{{x'=0}\})$, we have 

\begin{multline*} 
    \frac{4^k}{a_k} \left\|\frac{(\log|x'|)^k(x'\cdot \nabla_N)^k f}{|x'|^\frac{N}{2}}\right\|^2_{L^2(\mathbb{R}^n)}=  \left\|\frac{f}{|x'|^\frac{N}{2}}\right\|^2_{L^2(\mathbb{R}^n)} \\+ \sum_{m=1}^{k}\frac{O(k,m)}{a_k}\Bigg\|\sum_{l=0}^{m-1}S(m-1,l)\frac{(\log|x'|)^l(x'\cdot \nabla_N)^l f}{|x'|^\frac{N}{2}} \\+2\sum_{\kappa=1}^{m}S(m,\kappa)\frac{(\log|x'|)^\kappa(x'\cdot \nabla_N)^\kappa f}{|x'|^\frac{N}{2}}\Bigg\|^2_{L^2(\mathbb{R}^n)}, 
\end{multline*}
where $a_k$ is the {\bf \em square of the double factorial}, $O(k, m)$ is the coefficient that depends on {\bf \em Oblong numbers}, $S(m, \kappa)$ and $S(m-1,l)$ are {\bf \em Stirling numbers of second kind}.\\
\item(\textbf{$L^{2}$-Higher order critical Sobolev type inequality})
Let $x=(x',x'') \in \mathbb{R}^N\times\mathbb{R}^{n-N}, 1\leq N\leq n,$ and $k\in \mathbb{N}.$  Then for any complex-valued $f \in C^\infty_0(\mathbb{R}^n\backslash\{{x'=0}\})$, we have 
\begin{equation*} 
  \left\|\frac{f}{|x'|^\frac{N}{2}}\right\|_{L^2(\mathbb{R}^n)}\leq \frac{2^k}{a(k)} \left\|\frac{(\log|x'|)^k(x'\cdot \nabla_N)^k f}{|x'|^\frac{N}{2}}\right\|_{L^2(\mathbb{R}^n)},
\end{equation*}
where $a(k)=(2k-1)!!$ is the double factorial of an odd number, and $\frac{2^k}{a(k)}$ is the sharp constant which is non-attainable. 
    \item(\textbf{Critical Caffarelli–Kohn–Nirenberg type inequality})
    Let $x=(x',x'') \in \mathbb{R}^N\times\mathbb{R}^{n-N}$. Let $1<p,q<\infty$, $0<r<\infty$ with $p+q\geq r$ and $\delta \in [0, 1] \cap \left[\frac{r-q}{r},\frac{p}{r}\right]$ and $b,c \in \mathbb{R}$. Assume that 
$$\frac{\delta r}{p}+\frac{(1-\delta)r}{q}=1 \ \ \text{and}  \ \ c=-\frac{N}{p}\delta+b(1-\delta).$$
Then we have the following critical Caffarelli-Kohn-Nirenberg type inequality for all $f\in C^\infty_0(\mathbb{R}^n \backslash \{{0}\})$:    
\begin{equation}\label{crit_CKN_intro}
\||x'|^c f\|_{L^r(\mathbb{R}^n)}\leq p^\delta \||x'|^{-\frac{N}{p}} \log|x'| (x'\cdot\nabla_N)f \|^\delta_{L^p(\mathbb{R}^n)}\||x'|^bf\|^{1-\delta}_{L^q(\mathbb{R}^n)},
\end{equation}
where $|x'|$ is the Euclidean norm on $\mathbb{R}^N,$ $\nabla_N$ is the standard gradient on $\mathbb{R}^N$. The constant $p$ is sharp for $p=q$ with $\frac{N}{p}=-b$ or for $\delta=0,1.$ 
\item(\textbf{Critical uncertainty type principle})
    Let $x=(x',x'') \in \mathbb{R}^N\times\mathbb{R}^{n-N}, 1\leq N\leq n$.  Then for any $f \in C^\infty_0(\mathbb{R}^n\backslash\{{x'=0}\})$ and $\frac{1}{N}+\frac{1}{q}=1,$ we have
    \begin{equation*} 
\int_{\mathbb{R}^n}|f|^2dx\leq N\left\||x'|^{-1}\log|x'|(x'\cdot\nabla_N) f\right\|_{L^N(\mathbb{R}^n)} \left\||x'|f\right\|_{L^q(\mathbb{R}^n)}  .
    \end{equation*}
\end{itemize} 

Note that the constants in the critical $L^{p}$-Sobolev type \eqref{sob in intro} and Caffarelli-Kohn-Nirenberg type \eqref{crit_CKN_intro} inequalities do not depend on the topological dimension $n$ compared to their classical counterparts. All the above results are also obtained on stratified Lie groups for a horizontal gradient with an Euclidean norm on the first stratum of the group and on a general homogeneous Lie group for the radial derivative operator for any homogeneous quasi-norm.

Section 2 provides a quick overview of the key ideas related to stratified Lie groups and homogeneous Lie groups, and establishes the terminology. Section 3 demonstrates the critical Sobolev type inequality and identities. In Section 4, a higher-order critical Sobolev type identity is shown. In Section 5, as an application, Caffarelli-Kohn-Nirenberg type inequalities and an uncertainty principle with logarithmic weights are proved.

\section{Preliminaries}This section provides a brief overview of the notations used for stratified Lie groups and for homogeneous Lie groups.

A Lie group $\mathbb{G} = (\mathbb{R}^n, \circ)$ is called a stratified group (or a homogeneous Carnot group) if it satisfies the following criteria:
\begin{itemize}
\item There is a decomposition of $\mathbb{R}^n$ into $\mathbb{R}^{N}, \mathbb{R}^{N_2}, ..., \mathbb{R}^{N_r},$ where the sum of the dimensions equals $n$, $\mathbb{R}^{N}=\mathbb{R}^{N_1}$ and the dilation $\delta_{\lambda}:\mathbb{R}^{n}\rightarrow \mathbb{R}^{n}$ is an automorphism of the group $\mathbb{G}$. 
$$\delta_{\lambda}(x) = \delta_\lambda \left(x', x^{(2)}, \dots, x^{(r)} \right) := \left(\lambda x', \lambda^2 x^{(2)}, \dots, \lambda^r x^{(r)} \right)$$ for every $\lambda > 0$. Here $x'\equiv x^{(1)} \in \mathbb{R}^N$ and $x^{(k)} \in \mathbb{R}^{N_k}$
for $k = 2, \dots, r$.
\item 
Let $N$ be as previously stated and let $X_1, \dots, X_N$ be the left invariant vector fields on $\mathbb{G}$ such that $X_k(0) = \frac{\partial}{\partial x_k}|_0$ for $k = 1, \dots, N$. Then the rank of the Lie bracket of $X_1, \dots, X_N$ is equal to $n$ for every $x \in \mathbb{R}^n$, meaning that the iterated Lie brackets of $X_1, \dots, X_N$ span the Lie algebra of $\mathbb{G}$.
\end{itemize}
The group is defined as $\mathbb{G} = (\mathbb{R}^n, \circ, \delta_{\lambda})$. This kind of groups have been comprehensively investigated by Folland \cite{FOL75}. The left-invariant vector fields $X_1, ..., X_N$ are known as the (Jacobian) generators of the group and $r$ is referred to as the step of the group. The homogeneous dimension of the group is represented by $Q = \sum_{k=1}^r k N_k$, $N_1 = N$. The standard Lebesgue measure $dx$ on $\mathbb{R}^n$ serves as the Haar measure for the group (see, e.g. [\cite{FR16}, Proposition
1.6.6]). For more details on stratified Lie groups we refer to \cite{BLU07} or \cite{FR16}. The notation $\nabla_H$ represents the horizontal gradient, $\operatorname{div}_H\upsilon$ represents the horizontal divergence, and $|x'|$ represents the Euclidean norm on $\mathbb{R}^{N}$.

The left-invariant vector fields $X_j$ can be written as in formula (\ref{inv.vec.fields}) and satisfy the divergence theorem,  see e.g. \cite{RS17}  for the derivation of the exact formula, see also [\cite{FR16}, Section 3.1.5] for a general presentation. 
\begin{equation} \label{inv.vec.fields}
X_k = \frac{\partial}{\partial x'_k} + \sum_{l=2}^{r} \sum_{m=1}^{N_l} a^{(l)}_{k,m} (x', \dots, x^{l-1}) \frac{\partial}{\partial x_m^{(l)}}.    
\end{equation}
Formulae (\ref{hor.gradient}) and (\ref{div}) define the horizontal gradient and the horizontal divergence, respectively, 
\begin{equation} \label{hor.gradient}
    \nabla_H:=(X_1, \dots, X_N),
\end{equation}

\begin{equation} \label{div}
    \operatorname{div}_Hv:=\nabla_H\cdot v.
\end{equation}
 We will also use
 $|x'|=\sqrt{{x'_1}^2+\dots + {x'_N}^2}$ for the Euclidean norm on $\mathbb{R}^N$, and 
 
Now, let us give some necessary notations
concerning a general homogeneous Lie group following Folland and Stein \cite{FS82} (see also \cite[Section 3.1.7]{FR16} and \cite{RS19}).

A homogeneous quasi-norm on $\mathbb{G}$ is a continuous and non-negative function 
$$\mathbb{G}\ni x \rightarrow |x|\in [0,\infty),$$
which satisfies the following properties
\begin{itemize}
    \item $|x^{-1}|=|x|$ for all $x\in \mathbb{G},$
    \item $|\lambda x|=\lambda|x|$ for all $x\in \mathbb{G}$ and $\lambda>0,$
    \item $|x|=0$ if and only if $x=0.$
\end{itemize}

The polar decomposition, which is used in this paper, can be described as follows: there exists a unique positive Borel measure $\sigma$ defined on the unit sphere
$$\mathfrak{S}  :=\{x\in \mathbb{G}:|x|=1\},$$
such that for all $f\in L^1(\mathbb{G})$ we have
$$\int_\mathbb{G} f(x)dx=\int_0^\infty \int_\mathfrak{S}  f(ry) r^{Q-1}d\sigma (y)dr.$$

We will also use the radial derivative operator
$$\mathcal{R}_{|x|}f(x):=\frac{d}{d|x|}f(x)$$
with respect to any homogeneous quasi-norm $|x|$ on a homogeneous Lie group $\mathbb{G}.$

\section{$L^{p}$-Critical Sobolev type (improved Hardy) inequalities and identities}
In this section, we investigate the critical case of the Sobolev type (improved Hardy) inequality and identities on $\mathbb{R}^n$, stratified Lie groups, and homogeneous Lie groups. As a special case, we also obtain the critical case of the extended Hardy inequality and its stability result.
\vspace{4pt}

    \begin{thm} \label{th2.1}
    
    Let $x=(x',x'') \in \mathbb{R}^N\times\mathbb{R}^{n-N}, 1\leq N\leq n$, $1<p<\infty$, $|x'|$ is the Euclidean norm on $\mathbb{R}^N$ and $\nabla_N$ is the standard gradient on $\mathbb{R}^N$. 
    \begin{itemize}
        \item[(i)] Then for any complex-valued function $f \in C^\infty_0(\mathbb{R}^n\backslash\{{x'=0}\})$ we have 
        \begin{equation} \label{sob}
        \left\|{\frac{f}{|x'|^\frac{N}{p} }}\right\|_{L^p(\mathbb{R}^n)}
        \leq p \left\|\frac{\log|x'|(x'\cdot \nabla_N) f}{|x'|^\frac{N}{p}}\right\|_{L^p(\mathbb{R}^n)},
        \end{equation}
        where $p$ is sharp.
        \item[(ii)]
        For any complex-valued function $f \in C^\infty_0(\mathbb{R}^n\backslash\{{x'=0}\})$  we have the following identity, which is the sharp remainder formula for (\ref{sob}):
\begin{multline} \label{Cp rem. of crit. sob.}
    \left\|\frac{f}{|x'|^\frac{N}{p}}\right\|_{L^p(\mathbb{R}^n)}^p=p^p\left\|\frac{\log|x'|(x'\cdot \nabla_N) f}{|x'|^\frac{N}{p}} \right\|_{L^p(\mathbb{R}^n)}^p-\\
    -\int_{\mathbb{R}^n}C_p\left(\frac{p\log|x'|(x'\cdot \nabla_N) f}{|x'|^\frac{N}{p}} ,\frac{p(\log|x'|)^{-\frac{1}{p}+1}}{|x'|^\frac{N}{p}}(x'\cdot\nabla_N)\left(f(\log|x'|)^\frac{1}{p}\right)\right)dx,
\end{multline}
where the functional $C_p$ is given by
\begin{equation}\label{Cp}
C_p(u,v):=|u|^p-|u-v|^p-p|u-v|^{p-2}{\rm{Re}}((u-v)\cdot\overline{v})\geq0.\end{equation}
Moreover, for $p\geq2$, $C_p$ vanishes if and only if $$f=(\log|x'|)^{-\frac{1}{p}}\varphi\left(\frac{x'}{|x'|},x''\right)$$ 
for some function $\varphi:\mathbb{S}^{N-1}\times\mathbb{R}^{n-N}\rightarrow\mathbb{C}$, which makes the left-hand side of \eqref{Cp rem. of crit. sob.} infinite unless $\varphi=f=0$ on the basis of the non-integrability of the function
$$\frac{|f|^p}{|x'|^N}=\frac{\left|\varphi\left(\frac{x'}{|x'|},x''\right)\right|^p}{|\log|x'|||x'|^N}$$
on $\mathbb{R}^n$. Consequently, the sharp constant $p$ in \eqref{sob} is not attained \\on $C^\infty_0(\mathbb{R}^n\backslash\{{x'=0}\}).$
    \end{itemize}  
    \end{thm} 
\begin{rem}
The inequality \eqref{sob} is the critical case of the inequality from \cite{KY24} when $\alpha=N$.
\end{rem}
    
    \begin{proof}[Proof of Theorem \ref{sob}]
Using the identity $ x'\cdot \nabla_N \log|x'|=1$ and 
   integrating by parts we have
\begin{align} 
    \int_{\mathbb{R}^n} \frac{|f(x)|^p}{|x'|^N}dx &=\int_{\mathbb{R}^n} \frac{|f(x)|^p}{|x'|^N} (x'\cdot \nabla_N \log|x'|)dx \nonumber\\
    &=\int_{\mathbb{R}^n} \frac{|f(x)|^p}{|x'|^N} \left(\sum_{j=1}^N x'_j\cdot \partial_{x'_j} \log|x'|\right)dx\nonumber\\ 
    &=-\int_{\mathbb{R}^n}\sum_{j=1}^N\left[\log|x'| x'_j \partial_{x'_j}\left(\frac{|f(x)|^p}{|x'|^N}\right)+\log|x'|\frac{|f(x)|^p}{|x'|^N}\right]dx\nonumber\\
    &=-\int_{\mathbb{R}^n}\sum_{j=1}^N \bigg[ \log|x'|{\rm Re}\left(pf(x)|f(x)|^{p-2}\overline{x'_j\partial_{x'_j}f(x)}\right)|x'|^{-N}+\nonumber\\
    &+\log|x'||f(x)|^p(-N)|x'|^{-N-2}{x'_j}^2 +\log|x'|\frac{|f(x)|^p}{|x'|^N}\bigg]dx\nonumber\\    
    &=-p\int_{\mathbb{R}^n}\sum_{j=1}^N{\rm Re}\left(f(x)|f(x)|^{p-2}\overline{x'_j \partial_{x'_j}f(x)}\right)|x'|^{-N}\log|x'|dx. \label{proof of sob}
\end{align} 
Here using the H\"{o}lder inequality we calculate
\begin{align*}
    \int_{\mathbb{R}^n} \frac{|f(x)|^p}{|x'|^N}dx &=-p\int_{\mathbb{R}^n}\sum_{j=1}^N{\rm Re}\left(f(x)|f(x)|^{p-2}\overline{x'_j \partial_{x'_j}f(x)}\right)|x'|^{-N}\log|x'|dx\\
    &\leq |-p|\int_{\mathbb{R}^n} \frac{|f(x)|^{p-1}}{|x'|^N} |\log|x'|(x'\cdot \nabla_N f(x))| dx \\
    &= p\int_{\mathbb{R}^n} \frac{|f(x)|^{p-1}}{|x'|^{N(\frac{p-1}{p})}} \frac{|\log|x'|(x'\cdot \nabla_N f(x)) |}{|x'|^\frac{N}{p}}dx \\
    &\leq p\left(\int_{\mathbb{R}^n} \frac{|f(x)|^p}{|x'|^N}dx\right)^\frac{p-1}{p} \left(\int_{\mathbb{R}^n} \frac{|\log|x'|(x'\cdot \nabla_N f(x)) |^p}{|x'|^N}dx\right)^\frac{1}{p},
\end{align*}    
 which implies the inequality \eqref{sob} in Part (i). Now let us prove that $p$ is sharp. We note that the function 
    $$h_1(x)=\big(\log|x'|\big)^{-\frac{1}{p}}$$
    satisfies the following H\"{o}lder equality condition
    $$p^p\frac{|(x'\cdot \nabla_N h_1(x)) \log|x'||^p}{|x'|^N}=\frac{|h_1(x)|^p}{|x'|^N},$$
showing the sharpness of the constant.    

We now prove Part (ii).
To prove identity \eqref{Cp rem. of crit. sob.} we use non-negative functional \eqref{Cp} from \cite[Lemma 3.3]{CKLL24}. By using notations
$$u=\frac{p\log|x'|(x'\cdot \nabla_N) f(x)}{|x'|^\frac{N}{p}}, \ \ v=\frac{p(\log|x'|)^{-\frac{1}{p}+1}}{|x'|^\frac{N}{p}}(x'\cdot\nabla_N)\left(f(x)(\log|x'|)^\frac{1}{p}\right)$$
in \eqref{Cp}, and using that $u-v=-\frac{f(x)}{|x'|^\frac{N}{p}},$ we get
\begin{multline} \label{Cp proof}
\int_{\mathbb{R}^n}C_p\left(\frac{p\log|x'|(x'\cdot \nabla_N) f(x)}{|x'|^\frac{N}{p}} ,\frac{p(\log|x'|)^{-\frac{1}{p}+1}}{|x'|^\frac{N}{p}}(x'\cdot\nabla_N)\left(f(x)(\log|x'|)^\frac{1}{p}\right)\right)dx\\=\int_{\mathbb{R}^n}\left|\frac{p\log|x'|(x'\cdot \nabla_N) f(x)}{|x'|^\frac{N}{p}}\right|^pdx-\int_{\mathbb{R}^n}\left|-\frac{f(x)}{|x'|^\frac{N}{p}}\right|^pdx\\-p\int_{\mathbb{R}^n}\left|-\frac{f(x)}{|x'|^\frac{N}{p}}\right|^{p-2}{\rm{Re}}\left(-\frac{f(x)}{|x'|^\frac{N}{p}}\cdot\overline{\left(\frac{p\log|x'|(x'\cdot \nabla_N) f(x)}{|x'|^\frac{N}{p}}+\frac{f(x)}{|x'|^\frac{N}{p}}\right)}\right)dx\\=\int_{\mathbb{R}^n}\left|\frac{p\log|x'|(x'\cdot \nabla_N) f(x)}{|x'|^\frac{N}{p}}\right|^pdx-\int_{\mathbb{R}^n}\left|\frac{f(x)}{|x'|^\frac{N}{p}}\right|^pdx
\\+p\left({\rm{Re}}\int_{\mathbb{R}^n}\frac{|f(x)|^{p-2}}{|x'|^N}f(x)\cdot\overline{p\log|x'|(x'\cdot \nabla_N) f(x)}dx+\int_{\mathbb{R}^n}\frac{|f(x)|^p}{|x'|^N}dx\right).\end{multline}
Using \eqref{proof of sob} in \eqref{Cp proof}, the last term in \eqref{Cp proof} vanishes, and we get \eqref{Cp rem. of crit. sob.}.

Now, to check the existence of nontrivial extremisers, we examine the sharp remainder term on the right-hand side of \eqref{Cp rem. of crit. sob.}.
From \cite[Step 3 of Proof of Lemma 3.4]{CKLL24} we know that for $p\geq2$ we have $$C_p(u,v):=|u|^p-|u-v|^p-p|u-v|^{p-2}{\rm{Re}}((u-v)\cdot\overline{v})\geq c_p|v|^p$$
for some constant $c_p\in(0,1]$. Thus, we can easily see that $C_p(u,v)=0$ if and only if $v=0.$ Therefore, the remainder term vanishes if $f$ satisfies $$v=\frac{(\log|x'|)^{-\frac{1}{p}+1}}{|x'|^\frac{N}{p}}(x'\cdot\nabla_N)\left(f(x)(\log|x'|)^\frac{1}{p}\right)=0.$$
By integrating 
$$\frac{(x'\cdot\nabla_N)}{|x'|}\left(f(x)(\log|x'|)^\frac{1}{p}\right)=0$$ we find that $f(x)=(\log|x'|)^{-\frac{1}{p}}\varphi\left(\frac{x'}{|x'|},x''\right),$ 
where $\varphi:\mathbb{S}^{N-1}\times\mathbb{R}^{n-N}\rightarrow\mathbb{C}$.
But, this makes the left-hand side of \eqref{Cp rem. of crit. sob.} infinite unless $f=\varphi=0$, which proves nonexistence of nontrivial extremisers. 

Thus, the proof of Theorem \ref{sob} is complete.
\end{proof}
Using the Schwarz inequality on the right-hand side of (\ref{sob}) we obtain the critical case of the Hardy type inequality from \cite{KY24}. Let us show this in the next corollary.

\begin{cor} \label{hardy}
Let $x=(x',x'') \in \mathbb{R}^N\times\mathbb{R}^{n-N}, 1\leq N\leq n.$  Then for any $f \in C^\infty_0(\mathbb{R}^n\backslash\{{x'=0}\})$, and all $1<p<\infty,$ we have
 \begin{equation} \label{c.hardy} 
    \left\|{\frac{f}{|x'|^\frac{N}{p} }}\right\|_{L^p(\mathbb{R}^n)}
    \leq p \left\|\frac{\nabla_N f}{|x'|^{\frac{N}{p}-1}}\log|x'|\right\|_{L^p(\mathbb{R}^n)},
    \end{equation}
where $|x'|$ is the Euclidean norm on $\mathbb{R}^N,$ $\nabla_N$ is the standard gradient on $\mathbb{R}^N,$ and $p$ is sharp.
\end{cor}
\begin{proof}[Proof of Corollary \ref{hardy}] We can write the inequality (\ref{sob}) in the following way:
    $$\left(\int_{\mathbb{R}^n}\left|\frac{f}{|x'|^\frac{N}{p}}\right|^pdx\right)^\frac{1}{p}
    \leq p \left(\int_{\mathbb{R}^n}\left|\frac{x'}{|x'|}\cdot\frac{\nabla_N f}{|x'|^{\frac{N}{p}-1}}\log|x'|\right|^pdx\right)^\frac{1}{p}.$$
    Then by the Schwarz inequality we obtain
    $$p \left(\int_{\mathbb{R}^n}\left|\frac{x'}{|x'|}\cdot\frac{\nabla_N f}{|x'|^{\frac{N}{p}-1}}\log|x'|\right|^pdx\right)^\frac{1}{p} \leq p \left(\int_{\mathbb{R}^n}\left|\frac{x'}{|x'|}\right|^p\left|\frac{\nabla_N f}{|x'|^{\frac{N}{p}-1}}\log|x'|\right|^pdx\right)^\frac{1}{p},$$
which implies (\ref{c.hardy}).   
\end{proof}Thus, (\ref{sob}) can be regarded as a refinement of (\ref{c.hardy}). \\

 In Theorem \ref{th2.1}, by taking $p=N$ and doing some calculations, we also obtain the critical case of the Hardy type inequality in the Badiale–Tarantello conjecture \cite{BT02}, and its stability result.
\begin{cor} \label{critical badialle}
    Let $x=(x',x'') \in \mathbb{R}^N\times\mathbb{R}^{n-N}, 2\leq N\leq n,$ $|x'|$ is the Euclidean norm on $\mathbb{R}^N.$ 
    \begin{itemize}
        \item[(i)] Then for any complex-valued function $f \in C^\infty_0(\mathbb{R}^n\backslash\{{x'=0}\})$ we have
\begin{equation} \label{c.hardy from badialle} 
    \left\|{\frac{f}{|x'|}}\right\|_{L^N(\mathbb{R}^n)}
    \leq N \left\|\log|x'|\nabla f\right\|_{L^N(\mathbb{R}^n)},
    \end{equation}
    where $N$ is sharp.
    \item[(ii)]
     For any complex-valued function $f \in C^\infty_0(\mathbb{R}^n\backslash\{{x'=0}\})$  we have the following inequality, which is the stability result for (\ref{c.hardy from badialle}):
\begin{multline} \label{rem. of crit. badialle.}
   N^N\left\|\log|x'|\nabla f \right\|_{L^N(\mathbb{R}^n)}^N \geq \left\|\frac{f}{|x'|}\right\|_{L^N(\mathbb{R}^n)}^N\\
   +\int_{\mathbb{R}^n}C_N\left(\frac{N\log|x'|(x'\cdot \nabla_N) f}{|x'|} ,\frac{N(\log|x'|)^{-\frac{1}{N}+1}}{|x'|}(x'\cdot\nabla_N)\left(f(\log|x'|)^\frac{1}{N}\right)\right)dx,
\end{multline}
where the functional $C_N$ is given by
\begin{equation*}
C_N(u,v):=|u|^N-|u-v|^N-N|u-v|^{N-2}{\rm{Re}}((u-v)\cdot\overline{v})\geq0.\end{equation*}
    \end{itemize}
\end{cor}
\begin{proof}[Proof of Corollary \ref{critical badialle}]
For $x'_0=(x',0)\in \mathbb{R}^n,$ one can get the identities:
\begin{equation} \label{x0}
    |x'|=|x'_0|, \ \ x'\nabla_N f(x) = x'_0\nabla f(x).
\end{equation} 
By taking $p=N$ in \eqref{sob} and using identities (\ref{x0}) and the Schwarz inequality, we obtain \eqref{c.hardy from badialle}.

Now let us prove \eqref{rem. of crit. badialle.}. In \eqref{Cp rem. of crit. sob.}, by taking $p=N$, using identities \eqref{x0} and then the Schwarz inequality yields
   
\begin{align}
   \int_{\mathbb{R}^n}C_N\Bigg( \frac{N\log|x'|(x'\cdot \nabla_N) f}{|x'|} ,\frac{N(\log|x'|)^{-\frac{1}{N}+1}}{|x'|}&(x'\cdot\nabla_N)\left(f(\log|x'|)^\frac{1}{N}\right) \Bigg) dx+\nonumber\\
    +\left\|\frac{f}{|x'|}\right\|_{L^N(\mathbb{R}^n)}^N&=N^N\left\|\frac{\log|x'|(x'\cdot \nabla_N) f}{|x'|} \right\|_{L^N(\mathbb{R}^n)}^N\nonumber\\ &=N^N\left\|\frac{\log|x'|(x'_0\cdot \nabla) f}{|x'_0|} \right\|_{L^N(\mathbb{R}^n)}^N\nonumber\\&\leq N^N\left\|\log|x'|\cdot \nabla f \right\|_{L^N(\mathbb{R}^n)}^N,
\end{align}
which implies \eqref{rem. of crit. badialle.}.
\end{proof}

Now, we present the results discussed above within the context of stratified Lie groups. The methods employed to obtain these results are not limited to the Euclidean setting; hence, we can extend our findings to the stratified Lie groups using the same approach. Consequently, for the sake of brevity, we will provide the following results without accompanying proofs.  
\begin{thm} \label{th6.1}
        Let $\mathbb{G}$ be a stratified group with N being the dimension of the first stratum, $Q$ is the homogeneous dimension of $\mathbb{G}$, $|x'|$ is the Euclidean norm on $\mathbb{R}^N,$ $\nabla_H$ is the horizontal gradient on $\mathbb{G}$ and $1<p<\infty.$ 
    \begin{itemize}
        \item[(i)] Then for any complex-valued function $f \in C^\infty_0(\mathbb{G}\backslash\{{x'=0}\})$ we have 
        \begin{equation} \label{sob.on.strt}
        \left\|{\frac{f}{|x'|^\frac{N}{p} }}\right\|_{L^p(\mathbb{G})}
        \leq p \left\|\frac{x'\cdot \nabla_H f}{|x'|^\frac{N}{p}}\log|x'|\right\|_{L^p(\mathbb{G})},
        \end{equation}
        where $p$ is sharp.
        \item[(ii)]
        For any complex-valued function $f \in C^\infty_0(\mathbb{G}\backslash\{{x'=0}\})$  we have the following identity, which is the sharp remainder formula for (\ref{sob.on.strt}):
\begin{multline} \label{rem. of crit. sob. on strtfd}
    \left\|\frac{f}{|x'|^\frac{N}{p}}\right\|_{L^p(\mathbb{G})}^p=p^p\left\|\frac{\log|x'|(x'\cdot \nabla_H) f}{|x'|^\frac{N}{p}} \right\|_{L^p(\mathbb{G})}^p-\\
    -\int_{\mathbb{G}}C_p\left(\frac{p\log|x'|(x'\cdot \nabla_H) f}{|x'|^\frac{N}{p}} ,\frac{p(\log|x'|)^{-\frac{1}{p}+1}}{|x'|^\frac{N}{p}}(x'\cdot\nabla_H)\left(f(\log|x'|)^\frac{1}{p}\right)\right)dx,
\end{multline}
where the functional $C_p$ is given by
\begin{equation*}
C_p(u,v):=|u|^p-|u-v|^p-p|u-v|^{p-2}{\rm{Re}}((u-v)\cdot\overline{v})\geq0.\end{equation*}
    \end{itemize}
    \end{thm}

\begin{rem}
The inequality \eqref{sob.on.strt} is the critical case of the inequality from \cite[Theorem 3.1]{RSY17a} when $\alpha = \frac{N}{p}$.
\end{rem}  

Using the Schwarz inequality on the right-hand side of \eqref{sob.on.strt}, we derive the critical case of the Hardy-type inequality presented in \cite[Corollary 3.3]{RSY17a}. We demonstrate this in the following corollary.

\begin{cor} \label{cor6.2}
Let $\mathbb{G}$ be a stratified group with N being the dimension of the first stratum. Then for any $f \in C^\infty_0(\mathbb{G}\backslash\{{x'=0}\})$, and all $1<p<\infty,$ we have
 \begin{equation} \label{c.hardy.onstrt} 
    \left\|{\frac{f}{|x'|^\frac{N}{p} }}\right\|_{L^p(\mathbb{G})}
    \leq p \left\|\frac{\nabla_H f}{|x'|^{\frac{N}{p}-1}}\log|x'|\right\|_{L^p(\mathbb{G})},
    \end{equation}
where $|x'|$ is the Euclidean norm on $\mathbb{R}^N$ and $\nabla_H$ is the horizontal gradient on $\mathbb{G}.$ Moreover, the constant $p$ is sharp.
\end{cor}
The inequality (\ref{sob.on.strt}) can be considered as an improvement of (\ref{c.hardy.onstrt}).
\\

By considering the inequality \eqref{sob.on.strt} with $p=N$ and performing certain calculations, we can deduce the critical case of the Hardy-type inequality in the Badiale–Tarantello conjecture \cite{BT02}, on the stratified Lie groups. Additionally, we obtain its corresponding stability result on the stratified Lie groups.
\begin{cor} \label{critical badialle on strtfd}
    Let $\mathbb{G}$ be a stratified group with $N\geq2$ being the dimension of the first stratum, $|x'|$ the Euclidean norm on $\mathbb{R}^N,$ $\nabla_H$ the horizontal gradient on $\mathbb{G}$.
    
    \begin{itemize}
        \item[(i)] Then for any complex-valued function $f \in C^\infty_0(\mathbb{G}\backslash\{{x'=0}\})$ we have
\begin{equation} \label{c.hardy from badialle on strtfd} 
    \left\|{\frac{f}{|x'|}}\right\|_{L^N(\mathbb{G})}
    \leq N \left\|\log|x'|\nabla_H f\right\|_{L^N(\mathbb{G})},
    \end{equation}
    where $N$ is sharp.
    \item[(ii)]
        For any complex-valued function $f \in C^\infty_0(\mathbb{G}\backslash\{{x'=0}\})$, we have 
\begin{multline} \label{rem. of crit. badialle. on strtfd}
   N^N\left\|\log|x'|\nabla_H f \right\|_{L^N(\mathbb{G})}^N \geq \left\|\frac{f}{|x'|}\right\|_{L^N(\mathbb{G})}^N\\
   +\int_{\mathbb{G}}C_N\left(\frac{N\log|x'|(x'\cdot \nabla_H) f}{|x'|} ,\frac{N(\log|x'|)^{-\frac{1}{N}+1}}{|x'|}(x'\cdot\nabla_H)\left(f(\log|x'|)^\frac{1}{N}\right)\right)dx,
\end{multline}
where the functional $C_N$ is given by
\begin{equation*}
C_N(u,v):=|u|^N-|u-v|^N-N|u-v|^{N-2}{\rm{Re}}((u-v)\cdot\overline{v})\geq0.\end{equation*}
    \end{itemize}
\end{cor}

Now, let us present our results on a general homogeneous Lie group.
\begin{thm} \label{th.on hom.}
        Let $\mathbb{G}$ be a homogeneous group of homogeneous dimension $Q$ and let $|\cdot|$ be any homogeneous quasi-norm on $\mathbb{G}$. Let $\mathcal{R}_{|x|}=\frac{d}{d|x|}$ be the radial derivative and $1<p<\infty$. Then for any complex-valued function $f \in C^\infty_0(\mathbb{G}\backslash\{{0}\})$  we have the following identity, which is the sharp remainder formula for (\ref{sob.on.hom}):
\begin{multline} \label{rem. of crit. sob. on hom}
    \left\|\frac{f}{|x|^\frac{Q}{p}}\right\|_{L^p(\mathbb{G})}^p=\left\|p\frac{\mathcal{R}_{|x|} f}{|x|^{\frac{Q}{p}-1}} \log|x|\right\|_{L^p(\mathbb{G})}^p-\\
    -\int_{\mathbb{G}}C_p\left(p\frac{\mathcal{R}_{|x|} f}{|x|^{\frac{Q}{p}-1}} \log|x|,\frac{p(\log|x|)^{-\frac{1}{p}+1}}{|x|^{\frac{Q}{p}-1}}\mathcal{R}_{|x|}\left(f(\log|x|)^\frac{1}{p}\right)\right)dx,
\end{multline}
where the functional $C_p$ is given by
\begin{equation*}
C_p(u,v):=|u|^p-|u-v|^p-p|u-v|^{p-2}{\rm{Re}}((u-v)\cdot\overline{v})\geq0.\end{equation*}
Moreover, for $p\geq2$, $C_p$ vanishes if and only if $$f=(\log|x|)^{-\frac{1}{p}}\varphi\left(\frac{x}{|x|}\right)$$ 
for some function $\varphi:\mathfrak{S}\rightarrow\mathbb{C}$, which makes the left-hand side of \eqref{rem. of crit. sob. on hom} infinite unless $\varphi=f=0$ on the basis of the non-integrability of the function
$$\frac{|f|^p}{|x|^Q}=\frac{\left|\varphi\left(\frac{x}{|x|}\right)\right|^p}{|\log|x|||x|^Q}$$
on $\mathbb{G}$.
    \end{thm}
\begin{rem}
    By dropping the positive remainder term on the right-hand side of \eqref{rem. of crit. sob. on hom} we get the inequality from \cite[Theorem 3.4.]{RSY18b} 
    \begin{equation} \label{sob.on.hom}
        \left\|{\frac{f}{|x|^\frac{Q}{p} }}\right\|_{L^p(\mathbb{G})}
        \leq p \left\|\frac{\mathcal{R}_{|x|} f}{|x|^{\frac{Q}{p}-1}}\log|x|\right\|_{L^p(\mathbb{G})},
        \end{equation}
         where $|x|\mathcal{R}_{|x|} f=\mathbb{E} f$ is the Euler operator and $p$ is sharp. The last part of Theorem \ref{th.on hom.} implies that the sharp constant $p$ is not attained on $C^\infty_0(\mathbb{G}\backslash \{0\}).$
\end{rem}
    
    \begin{proof}[Proof of Theorem \ref{th.on hom.}]
        Using integration by parts, we have
\begin{align} 
    \int_{\mathbb{G}} \frac{|f(x)|^p}{|x|^Q}dx &=\int_0^\infty \int_\mathfrak{S} \frac{|f(ry)|^p}{r^{Q-Q+1}}d\sigma(y)dr \nonumber\\
    &=\int_0^\infty \int_\mathfrak{S}  |f(ry)|^p\left(\frac{d}{dr}{\rm \log} r\right) d\sigma(y)dr \nonumber\\ 
    &=-p{\rm Re}\int_0^\infty {\rm \log} r \int_\mathfrak{S}  |f(ry)|^{p-2} f(ry) \overline{\frac{df}{dr}} d\sigma(y)dr\nonumber\\
    &=-p{\rm Re}\int_\mathbb{G}\frac{|f(x)|^{p-2}f\overline{\mathcal{R}_{|x|} f(x)}}{|x|^{Q-1}}{\rm \log}|x|dx. \label{proof of sob. on hom}
\end{align} 
By using notations
$$u=\frac{p\log|x|\mathcal{R}_{|x|} f(x)}{|x|^{\frac{Q}{p}-1}}, \ \ v=\frac{p(\log|x|)^{-\frac{1}{p}+1}}{|x|^{\frac{Q}{p}-1}}\mathcal{R}_{|x|}\left(f(x)(\log|x|)^\frac{1}{p}\right)$$
in \eqref{Cp} we get that $u-v=-\frac{f(x)}{|x|^\frac{Q}{p}},$ and
\begin{multline} \label{Cp proof of hom}
\int_{\mathbb{G}}C_p\left(\frac{p\log|x|\mathcal{R}_{|x|} f(x)}{|x|^{\frac{Q}{p}-1}},\frac{p(\log|x|)^{-\frac{1}{p}+1}}{|x|^{\frac{Q}{p}-1}}\mathcal{R}_{|x|}\left(f(x)(\log|x|)^\frac{1}{p}\right)\right)dx\\=\int_{\mathbb{G}}\left|\frac{p\log|x|\mathcal{R}_{|x|} f(x)}{|x|^{\frac{Q}{p}-1}}\right|^pdx-\int_{\mathbb{G}}\left|-\frac{f(x)}{|x|^\frac{Q}{p}}\right|^pdx\\-p\int_{\mathbb{G}}\left|-\frac{f(x)}{|x|^\frac{Q}{p}}\right|^{p-2}{\rm{Re}}\left(-\frac{f(x)}{|x|^\frac{Q}{p}}\cdot\overline{\left(\frac{p\log|x|\mathcal{R}_{|x|} f(x)}{|x|^{\frac{Q}{p}-1}}+\frac{f(x)}{|x|^\frac{Q}{p}}\right)}\right)dx\\=\int_{\mathbb{G}}\left|\frac{p\log|x|\mathcal{R}_{|x|} f(x)}{|x|^{\frac{Q}{p}-1}}\right|^pdx-\int_{\mathbb{G}}\left|\frac{f(x)}{|x|^\frac{Q}{p}}\right|^pdx
\\+p\left(p{\rm{Re}}\int_{\mathbb{G}}\frac{|f(x)|^{p-2}}{|x|^{Q-1}}f(x)\cdot\overline{\log|x|\mathcal{R}_{|x|}f(x)}dx+\int_{\mathbb{G}}\frac{|f(x)|^p}{|x|^Q}dx\right).\end{multline}
Using \eqref{proof of sob. on hom} in \eqref{Cp proof of hom} implies \eqref{rem. of crit. sob. on hom}.

Now, to check the existence of nontrivial extremisers, we examine the sharp remainder term on the right-hand side of \eqref{rem. of crit. sob. on hom}.
Similar to the proof of Theorem \ref{th2.1}, we know that for $p\geq2$ $$C_p(u,v):=|u|^p-|u-v|^p-p|u-v|^{p-2}{\rm{Re}}((u-v)\cdot\overline{v})\geq c_p|v|^p,$$
where $c_p\in(0,1]$. Thus, we can easily see that $C_p(u,v)=0$ if and only if $v=0.$ Therefore, the remainder term vanishes if $f$ satisfies $$v=\frac{p(\log|x|)^{-\frac{1}{p}+1}}{|x|^{\frac{Q}{p}-1}}\mathcal{R}_{|x|}\left(f(x)(\log|x|)^\frac{1}{p}\right)=0.$$
By integrating 
$$\mathcal{R}_{|x|}\left(f(x)(\log|x|)^\frac{1}{p}\right)=0$$ we find that $f(x)=(\log|x|)^{-\frac{1}{p}}\varphi\left(\frac{x}{|x|}\right),$ 
where $\varphi:\mathfrak{S}\rightarrow\mathbb{C}$.
But, this makes the left-hand side of \eqref{rem. of crit. sob. on hom} infinite unless $f=\varphi=0$, which proves nonexistence of nontrivial extremisers. 
\end{proof} 
 In Theorem \ref{th.on hom.} by taking $p=Q$ we obtain a Hardy type identity and by dropping the positive remainder term we obtain the critical case of the Hardy type inequality by Badiale–Tarantello \cite{BT02} on homogeneous Lie groups.

\begin{rem} \label{critical badialle on hom}
    Let $\mathbb{G}$ be a homogeneous group of homogeneous dimension $Q$, and let $|\cdot|$ be any homogeneous quasi-norm on $\mathbb{G}$. Let $\mathcal{R}_{|x|}=\frac{d}{d|x|}$ be the radial derivative.  
    \begin{itemize}
        \item[(i)]
        For any complex-valued function $f \in C^\infty_0(\mathbb{G}\backslash\{{0}\})$, we have the following identity, which is the sharp remainder formula for (\ref{c.hardy from badialle on hom}):
\begin{multline} \label{rem. of crit. bad. on hom}
    \left\|\frac{f}{|x|}\right\|_{L^Q(\mathbb{G})}^Q=\left\|Q\log|x|\mathcal{R}_{|x|} f \right\|_{L^Q(\mathbb{G})}^Q-\\
    -\int_{\mathbb{G}}C_Q\left(Q \log|x|\mathcal{R}_{|x|} f,Q(\log|x|)^{-\frac{1}{Q}+1}\mathcal{R}_{|x|}\left(f(\log|x|)^\frac{1}{Q}\right)\right)dx,
\end{multline}
where the functional $C_Q$ is given by
\begin{equation*}\label{CQ}
C_Q(u,v):=|u|^Q-|u-v|^Q-Q|u-v|^{Q-2}{\rm{Re}}((u-v)\cdot\overline{v})\geq0.\end{equation*}
    \item[(ii)] For any complex-valued function $f \in C^\infty_0(\mathbb{G}\backslash\{{0}\})$ we have 
        \begin{equation} \label{c.hardy from badialle on hom}
        \left\|{\frac{f}{|x|}}\right\|_{L^Q(\mathbb{G})}
        \leq Q \left\|\log|x| \mathcal{R}_{|x|} f \right\|_{L^Q(\mathbb{G})},
        \end{equation}
        where $Q$ is sharp.
    \end{itemize}
\end{rem}

\section{$L^{2}$-Higher order identities}
Now, we investigate the $L^2$ higher-order version of the critical Sobolev type identity
\eqref{Cp rem. of crit. sob.}. 
For $p=2$, the identity \eqref{Cp rem. of crit. sob.} takes the form
\begin{multline} \label{rem. of crit. sob. L2}
    4\left\|\frac{\log|x'|(x'\cdot \nabla_N) f}{|x'|^\frac{N}{2}} \right\|_{L^2(\mathbb{R}^n)}^2-\left\|\frac{f}{|x'|^\frac{N}{2}}\right\|_{L^2(\mathbb{R}^n)}^2=\left\|\frac{f}{|x'|^\frac{N}{2}}+\frac{2\log|x'|(x'\cdot \nabla_N) f}{|x'|^\frac{N}{2}}\right\|_{L^2(\mathbb{R}^n)}^2,
\end{multline}
which by using $f=\log|x'|(x'\cdot\nabla_N)f$ in \eqref{rem. of crit. sob. L2} and iterating it $k-1$ times
gives
\begin{multline} \label{rem. of crit. sob. L2 higher order}
    4\left\|\frac{(\log|x'|(x'\cdot \nabla_N))^k f}{|x'|^\frac{N}{2}} \right\|_{L^2(\mathbb{R}^n)}^2-\left\|\frac{(\log|x'|(x'\cdot \nabla_N))^{k-1}f}{|x'|^\frac{N}{2}}\right\|_{L^2(\mathbb{R}^n)}^2\\=\left\|\frac{(\log|x'|(x'\cdot \nabla_N))^{k-1}f}{|x'|^\frac{N}{2}}+\frac{2(\log|x'|(x'\cdot \nabla_N))^k f}{|x'|^\frac{N}{2}}\right\|_{L^2(\mathbb{R}^n)}^2.
\end{multline}
By expanding the brackets inside the norms in \eqref{rem. of crit. sob. L2 higher order} and doing some calculations we can get the higher-order identity. Interestingly, we obtain this identity with the {\em Stirling numbers of the second
kind, Oblong numbers, and squares of the double factorial.}

Before stating the main result, let us introduce some notations:
\begin{equation} \label{sq.double}
a_k=(2k-1)!!^2,{\text{  
 }} k\in\mathbb{N},    
\end{equation}
is {\bf\em the square of the double factorial.}

For any non negative integers $k$ and $m$:
\begin{multline}
\label{coefficient}
O(k,m):=\left(\sum_{t=1}^{k-m}4^{k-t}a_t\sum_{t+1 \leq i_1<...<i_r\leq k-1} \prod_{j=1}^{r}o_{i_j}\right)+4^{m-1}a_{k-m+1},\ \ 1<m<k, 
\end{multline}
with $O(k,1)=a_k$ and $O(k,k)=4^{k-1}$, where $r=k-m-t+1$, and $o_{i_j}=(i_j)(i_j+1)$ is {\bf\em the Oblong number}. \\
We will also use the following {\bf\em Stirling number} of the second kind:
\begin{equation} \label{stirling}
S(m,\kappa)=\sum_{i=0}^{\kappa}\frac{(-1)^i}{\kappa!}\binom{\kappa}{i}(\kappa-i)^m,
\end{equation}
where $0\leq\kappa\leq m$, and $m$ is any non-negative integer.  
Recall that these numbers satisfy the recurrence relation  $$S(m+1,\kappa)=\kappa S(m,\kappa)+S(m,\kappa-1)$$ for $\kappa=1,2,...,m+1,$ $m=0,1,...,$ with initial conditions
$$S(0,0)=1, {\text{   }} S(m,0)=0 \text{ for } m>0,{\text{   }} S(m,\kappa)=0 \text{ for } \kappa>m.$$
For more details on such numbers we refer to \cite[Chapter 8]{CH02} and \cite[Section 2]{JM97}.\\

We can now state the theorem.
\begin{thm} \label{th4.1}
    Let $x=(x',x'') \in \mathbb{R}^N\times\mathbb{R}^{n-N}, 1\leq N\leq n,$ and $k\in \mathbb{N}.$  Then for any complex-valued $f \in C^\infty_0(\mathbb{R}^n\backslash\{{x'=0}\})$, we have 
\begin{multline} \label{higherorder}
    \frac{4^k}{a_k} \left\|\frac{(\log|x'|)^k(x'\cdot \nabla_N)^k f}{|x'|^\frac{N}{2}}\right\|^2_{L^2(\mathbb{R}^n)}=  \left\|\frac{f}{|x'|^\frac{N}{2}}\right\|^2_{L^2(\mathbb{R}^n)} \\+ \sum_{m=1}^{k}\frac{O(k,m)}{a_k}\Bigg\|\sum_{l=0}^{m-1}S(m-1,l)\frac{(\log|x'|)^l(x'\cdot \nabla_N)^l f}{|x'|^\frac{N}{2}} \\+2\sum_{\kappa=1}^{m}S(m,\kappa)\frac{(\log|x'|)^\kappa(x'\cdot \nabla_N)^\kappa f}{|x'|^\frac{N}{2}}\Bigg\|^2_{L^2(\mathbb{R}^n)}, 
\end{multline}
where $a_k=(2k-1)!!^2$, $O(k, m)>0$ and $S(m, \kappa)$ are defined in  (\ref{sq.double}), (\ref{coefficient}) and  (\ref{stirling}) respectively.
\end{thm} 

\begin{rem}
When $k=1$, we obtain identity \eqref{rem. of crit. sob. L2} which is the $L^2$ stability result of the critical Sobolev type inequality \eqref{sob}.
\end{rem}

\begin{cor} \label{higher order inequality}
     By dropping the sum of $L^2$ norms on the right-hand side of \eqref{higherorder} and taking the square root of it, we obtain the following $L^2$ higher order version of the inequality \eqref{sob}:    Let $x=(x',x'') \in \mathbb{R}^N\times\mathbb{R}^{n-N}, 1\leq N\leq n,$ and $k\in \mathbb{N}.$  Then for any complex-valued $f \in C^\infty_0(\mathbb{R}^n\backslash\{{x'=0}\})$, we have 
\begin{equation} \label{higherorder inequality}
  \left\|\frac{f}{|x'|^\frac{N}{2}}\right\|_{L^2(\mathbb{R}^n)}\leq \frac{2^k}{a(k)} \left\|\frac{(\log|x'|)^k(x'\cdot \nabla_N)^k f}{|x'|^\frac{N}{2}}\right\|_{L^2(\mathbb{R}^n)},
\end{equation}
where $a(k)=(2k-1)!!$ is the double factorial of an odd number, and $\frac{2^k}{a(k)}$ is the sharp constant which is non-attainable. 
\end{cor}
Now, let us state the next two lemmata, which will be useful in the proof of Theorem \ref{th4.1}.

\begin{lemma} \label{lemma 1.} 
Let $x=(x',x'') \in \mathbb{R}^N\times\mathbb{R}^{n-N}, 1\leq N\leq n,$ and $\kappa\in \mathbb{N}.$  Then for any $f \in C^\infty_0(\mathbb{R}^n\backslash\{{x'=0}\})$, we have
\begin{multline} \label{lemma 1}
(\log|x'|)^\kappa(x'\cdot \nabla_N)^\kappa \big(\log|x'|(x'\cdot \nabla_N) f \big)
\\=\kappa (\log|x'|)^\kappa(x'\cdot \nabla_N)^\kappa f+(\log|x'|)^{\kappa+1}(x'\cdot \nabla_N)^{\kappa+1}f,
\end{multline}
where $\nabla_N$ is the standard gradient on $\mathbb{R}^N.$ 
\end{lemma}

In the next lemma, we establish the recurrence relation for $O(k,m)$ in (\ref{coefficient}), which will be used to combine the remainder terms arising from the iteration process.
\begin{lemma} \label{lemma 2}
Let $k,m \in\mathbb{N},$ $2\leq m\leq k$. Then, we have 
\begin{equation} \label{recurrence}
4k(k+1)O(k,m)+4O(k,m-1)=O(k+1,m).
\end{equation}
\end{lemma}
The proofs of the Lemma \ref{lemma 1.} and Lemma \ref{lemma 2} are provided in the Appendix.
\begin{proof}[Proof of Theorem \ref{th4.1}.] 
We prove it by induction. 

$k=1.$ Equality (\ref{higherorder}) reduces to the identity (\ref{rem. of crit. sob. L2}), which has been proven in the previous section.

By assuming that identity (\ref{higherorder}) holds for $k^{th}$ order, let us prove it for $(k+1)^{th}$ order.

$(k+1)^{th}$ order. We have to prove that
\begin{multline} \label{higherorder k+1}
4^{k+1} \left\|\frac{(\log|x'|)^{k+1}(x'\cdot \nabla_N)^{k+1} f}{|x'|^\frac{N}{2}}\right\|^2_{L^2(\mathbb{R}^n)}= a_{k+1} \left\|\frac{f}{|x'|^\frac{N}{2}}\right\|^2_{L^2(\mathbb{R}^n)} \\+ \sum_{m=1}^{k+1}O({k+1},m)\Bigg\|\sum_{l=0}^{m-1}S(m-1,l)\frac{(\log|x'|)^l(x'\cdot \nabla_N)^l f}{|x'|^\frac{N}{2}} \\ \underbrace{+2\sum_{\kappa=1}^{m}S(m,\kappa)\frac{(\log|x'|)^\kappa(x'\cdot \nabla_N)^\kappa f}{|x'|^\frac{N}{2}}\Bigg\|^2_{L^2(\mathbb{R}^n)}}_{A^{m-1,m}}. 
\end{multline}
Replacing $f$ in (\ref{higherorder}) by $\log|x'|(x'\cdot \nabla_N) f$ gives
\begin{multline} \label{id.1}
4^k\left\|\frac{(\log|x'|)^k(x'\cdot \nabla_N)^k (\log|x'|(x'\cdot \nabla_N) f)}{|x'|^\frac{N}{2}}\right\|^2_{L^2(\mathbb{R}^n)}= a_k \left\|\frac{\log|x'|(x'\cdot \nabla_N) f}{|x'|^\frac{N}{2}}\right\|^2_{L^2(\mathbb{R}^n)}\\+ \sum_{m=1}^{k}O(k,m)\Bigg\|\sum_{l=0}^{m-1}S(m-1,l)\frac{(\log|x'|)^l(x'\cdot \nabla_N)^l( \log|x'|(x'\cdot \nabla_N) f)}{|x'|^\frac{N}{2}} \\  +2\sum_{\kappa=1}^{m}S(m,\kappa)\frac{(\log|x'|)^\kappa(x'\cdot \nabla_N)^\kappa (\log|x'|(x'\cdot \nabla_N) f)}{|x'|^\frac{N}{2}}\Bigg\|^2_{L^2(\mathbb{R}^n)}. 
\end{multline}

We demonstrate that identity (\ref{id.1}) implies identity (\ref{higherorder k+1}). We proceed to prove (\ref{higherorder k+1}) part by part. Firstly, we focus on establishing the second $L^2$ norm on the right-hand side of (\ref{higherorder k+1}), denoted as $A^{m-1,m}$. Since there are two similar terms within this norm, we will begin by handling the second term. The first term can be derived using the same method. 

The expression 
\begin{equation}    
\sum_{\kappa=1}^{m}S(m,\kappa)\frac{(\log|x'|)^\kappa(x'\cdot \nabla_N)^\kappa (\log|x'|(x'\cdot \nabla_N) f)}{|x'|^\frac{N}{2}} 
\end{equation}
in (\ref{id.1}) can be simplified by using Lemma \ref{lemma 1.}, as follows:

\begin{multline} \label{stirling reccurence}
\sum_{\kappa=1}^{m}S(m,\kappa)\frac{(\log|x'|)^\kappa(x'\cdot \nabla_N)^\kappa \big(\log|x'|(x'\cdot \nabla_N\big) f)}{|x'|^\frac{N}{2}}\\=\sum_{\kappa=1}^{m}S(m,\kappa)\left(\frac{\kappa (\log|x'|)^\kappa(x'\cdot \nabla_N)^\kappa f}{|x'|^\frac{N}{2}}+\frac{(\log|x'|)^{\kappa+1}(x'\cdot \nabla_N)^{\kappa+1}f}{|x'|^\frac{N}{2}}\right)\\=S(m,1)\left(\frac{ \log|x'|(x'\cdot \nabla_N) f}{|x'|^\frac{N}{2}}+\frac{(\log|x'|)^2(x'\cdot \nabla_N)^2 f}{|x'|^\frac{N}{2}}\right)\\+S(m,2)\left(\frac{ 2(\log|x'|)^2(x'\cdot \nabla_N)^2 f}{|x'|^\frac{N}{2}}+\frac{(\log|x'|)^3(x'\cdot \nabla_N)^3 f}{|x'|^\frac{N}{2}}\right)\\+S(m,3)\left(\frac{ 3(\log|x'|)^3(x'\cdot \nabla_N)^3 f}{|x'|^\frac{N}{2}}+\frac{(\log|x'|)^4(x'\cdot \nabla_N)^4 f}{|x'|^\frac{N}{2}}\right)\\...+S(m,m)\left(\frac{ m(\log|x')^m|(x'\cdot \nabla_N)^m f}{|x'|^\frac{N}{2}}+\frac{(\log|x'|)^{m+1}(x'\cdot \nabla_N)^{m+1} f}{|x'|^\frac{N}{2}}\right)
\end{multline}
 
\begin{multline}\label{stirling reccurence2}
=S(m,1)\frac{ \log|x'|(x'\cdot \nabla_N) f}{|x'|^\frac{N}{2}}+\bigg(S(m,1)+2S(m,2)\bigg)\frac{(\log|x'|)^2(x'\cdot \nabla_N)^2 f}{|x'|^\frac{N}{2}}\\+\bigg(S(m,2)+3S(m,3)\bigg)\frac{(\log|x'|)^3(x'\cdot \nabla_N)^3 f}{|x'|^\frac{N}{2}}\\+\bigg(S(m,3)+4S(m,4)\bigg)\frac{(\log|x'|)^4(x'\cdot \nabla_N)^4 f}{|x'|^\frac{N}{2}}\\...+\bigg(S(m,m-1)+mS(m,m)\bigg)\frac{(\log|x'|)^m(x'\cdot \nabla_N)^m f}{|x'|^\frac{N}{2}}\\+S(m,m)\frac{(\log|x'|)^{m+1}(x'\cdot \nabla_N)^{m+1} f}{|x'|^\frac{N}{2}}\\=\sum_{\kappa=1}^{m+1}\bigg(S(m,\kappa-1)+\kappa S(m,\kappa)\bigg)\frac{(\log|x'|)^\kappa(x'\cdot \nabla_N)^\kappa f}{|x'|^\frac{N}{2}}.
\end{multline}
Now by using the recurrence relation of Stirling numbers from \eqref{stirling} in \eqref{stirling reccurence2}
we get
\begin{multline} \label{rec4}
\sum_{\kappa=1}^{m}S(m,\kappa)\frac{(\log|x'|)^\kappa(x'\cdot \nabla_N)^\kappa \big(\log|x'|(x'\cdot \nabla_N\big) f)}{|x'|^\frac{N}{2}}\\=\sum_{\kappa=1}^{m+1}S(m+1,\kappa)\frac{(\log|x'|)^\kappa(x'\cdot \nabla_N)^\kappa f}{|x'|^\frac{N}{2}}.
\end{multline}
By following the same method, we can simplify the first term in the second $L^2$ norm on the right-hand side of (\ref{id.1}), as follows:
\begin{multline}
    \sum_{l=0}^{m-1}S(m-1,l)\frac{(\log|x'|)^l(x'\cdot \nabla_N)^l( \log|x'|(x'\cdot \nabla_N) f)}{|x'|^\frac{N}{2}}\\=\sum_{l=0}^{m-1}S(m-1,l)\left(\frac{l(\log|x'|)^l(x'\cdot \nabla_N)^l f}{|x'|^\frac{N}{2}}+\frac{(\log|x'|)^{l+1}(x'\cdot \nabla_N)^{l+1}f}{|x'|^\frac{N}{2}}\right) \\=\sum_{l=0}^{m}\bigg(l S(m-1,l)+S(m-1,l-1)\bigg)\frac{(\log|x'|)^l(x'\cdot \nabla_N)^l f}{|x'|^\frac{N}{2}}, 
\end{multline}
which implies
\begin{equation} \label{rec5}
    \sum_{l=0}^{m-1}S(m-1,l)\frac{(\log|x'|)^l(x'\cdot \nabla_N)^l( \log|x'|(x'\cdot \nabla_N) f)}{|x'|^\frac{N}{2}}=\sum_{l=0}^{m}S(m,l)\frac{(\log|x'|)^l(x'\cdot \nabla_N)^l f}{|x'|^\frac{N}{2}}.
\end{equation}
By using \eqref{rec4} and \eqref{rec5} in the second $L^2$ norm on right-hand side of (\ref{id.1}) we get
\begin{multline} \label{A.m.m+1}
\Bigg\|\sum_{l=0}^{m-1}S(m-1,l)\frac{(\log|x'|)^l(x'\cdot \nabla_N)^l( \log|x'|(x'\cdot \nabla_N) f)}{|x'|^\frac{N}{2}} \\  +2\sum_{\kappa=1}^{m}S(m,\kappa)\frac{(\log|x'|)^\kappa(x'\cdot \nabla_N)^\kappa (\log|x'|(x'\cdot \nabla_N) f)}{|x'|^\frac{N}{2}}\Bigg\|^2_{L^2(\mathbb{R}^n)}\\=\underbrace{\Bigg\|\sum_{l=0}^{m}S(m,l)\frac{(\log|x'|)^l(x'\cdot \nabla_N)^l f}{|x'|^\frac{N}{2}}  +2\sum_{\kappa=1}^{m+1}S(m+1,\kappa)\frac{(\log|x'|)^\kappa(x'\cdot \nabla_N)^\kappa f}{|x'|^\frac{N}{2}}\Bigg\|^2_{L^2(\mathbb{R}^n)}}_{A^{m,m+1}},    
\end{multline}
which we denote as $A^{m,m+1}$. Later in (\ref{id.11}), by shifting the index $m$, we will show how to derive $A^{m-1,m}$ which is from (\ref{higherorder k+1}).

Now let us show the left-hand side of (\ref{higherorder k+1}). The $L^2$ norm on the left-hand side of (\ref{id.1}) can be simplified by using Lemma \ref{lemma 1.} as follows:
\begin{multline} \label{id.2}
   \left\|\frac{(\log|x'|)^k(x'\cdot \nabla_N)^k (\log|x'|(x'\cdot \nabla_N) f)}{|x'|^\frac{N}{2}}\right\|^2_{L^2(\mathbb{R}^n)}\\=\left\|\frac{k(\log|x'|)^k(x'\cdot \nabla_N)^kf+ (\log|x'|)^{k+1}(x'\cdot \nabla_N)^{k+1} f)}{|x'|^\frac{N}{2}}\right\|^2_{L^2(\mathbb{R}^n)} \\=k^2\left\|\frac{(\log|x'|)^k(x'\cdot \nabla_N)^kf}{|x'|^\frac{N}{2}}\right\|^2_{L^2(\mathbb{R}^n)}+\left\|\frac{(\log|x'|)^{k+1}(x'\cdot \nabla_N)^{k+1}f}{|x'|^\frac{N}{2}}\right\|^2_{L^2(\mathbb{R}^n)}\\+2{\rm Re}\int_{\mathbb{R}^n}\frac{k(\log|x'|)^{2k+1}}{|x'|^N}(x'\cdot \nabla_N)^{k}f\overline{(x'\cdot\nabla_N)^{k+1}f}dx.
\end{multline}
To simplify the last term of the right-hand side of \eqref{id.2} we use integration by parts as follows:
\begin{multline*}
    2{\rm Re}\int_{\mathbb{R}^n} \frac{k(\log|x'|)^{2k+1}}{|x'|^N}(x'\cdot \nabla_N)^{k}f\overline{(x'\cdot\nabla_N)^{k+1}f}dx\\=2{\rm Re}\int_{\mathbb{R}^n} \left[\frac{k(\log|x'|)^{2k+1}}{|x'|^N}(x'\cdot \nabla_N)^{k}f\right](x'\cdot\nabla_N)\overline{(x'\cdot\nabla_N)^{k}f}dx\\=2{\rm Re}\int_{\mathbb{R}^n} \left[\frac{k(\log|x'|)^{2k+1}}{|x'|^N}(x'\cdot \nabla_N)^{k}f\right]\sum_{j=1}^{N}x'_j\partial_{x'_j}\overline{(x'\cdot\nabla_N)^{k}f}dx\\=-2{\rm Re}\int_{\mathbb{R}^n} \sum_{j=1}^{N}\Bigg[(2k^2+k)(\log|x'|)^{2k}\frac{\frac{2x'_j}{2|x'|}}{|x'|}|x'|^{-N}(x'\cdot \nabla_N)^{k}f x'_j\\+k(\log|x'|)^{2k+1}(-N)|x'|^{-N-1}\frac{2x'_j}{2|x'|}(x'\cdot \nabla_N)^{k}f x'_j+\frac{k(\log|x'|)^{2k+1}}{|x'|^N}(x'_j\partial_{x'_j})(x'\cdot \nabla_N)^{k}f\\+\frac{k(\log|x'|)^{2k+1}}{|x'|^N}(x'\cdot \nabla_N)^{k}f\Bigg]\overline{(x'\cdot\nabla_N)^{k}f}dx\\=-2{\rm Re}\int_{\mathbb{R}^n} \Bigg[\frac{(2k^2+k)(\log|x'|)^{2k}}{|x'|^N}(x'\cdot \nabla_N)^{k}f -N\frac{k(\log|x'|)^{2k+1}}{|x'|^N}(x'\cdot \nabla_N)^{k}f \\+\frac{k(\log|x'|)^{2k+1}}{|x'|^N}(x'\cdot \nabla_N)^{k+1}f+N\frac{k(\log|x'|)^{2k+1}}{|x'|^N}(x'\cdot \nabla_N)^{k}f\Bigg]\overline{(x'\cdot\nabla_N)^{k}f}dx\\=-2{\rm Re}\int_{\mathbb{R}^n} \frac{(2k^2+k)(\log|x'|)^{2k}}{|x'|^N}(x'\cdot \nabla_N)^{k}f \overline{(x'\cdot\nabla_N)^{k}f}dx \\-2{\rm Re}\int_{\mathbb{R}^n}\frac{k(\log|x'|)^{2k+1}}{|x'|^N}(x'\cdot \nabla_N)^{k+1}f\overline{(x'\cdot\nabla_N)^{k}f}dx
\end{multline*}
\begin{multline*}
=-2\int_{\mathbb{R}^n} \frac{(2k^2+k)(\log|x'|)^{2k}}{|x'|^N}|(x'\cdot\nabla_N)^{k}f|^2dx\\-2{\rm Re}\int_{\mathbb{R}^n} \frac{k(\log|x'|)^{2k+1}}{|x'|^N}(x'\cdot \nabla_N)^{k+1}f\overline{(x'\cdot\nabla_N)^{k}f}dx.
\end{multline*}
It means that
\begin{multline}
    \label{id.3}
    2{\rm Re}\int_{\mathbb{R}^n} \frac{k(\log|x'|)^{2k+1}}{|x'|^N}(x'\cdot \nabla_N)^{k}f\overline{(x'\cdot\nabla_N)^{k+1}f}dx\\=-(2k^2+k)\left\| \frac{(\log|x'|)^k(x'\cdot\nabla_N)^{k}f}{|x'|^\frac{N}{2}}\right\|^2_{L^2(\mathbb{R}^n)}.
\end{multline}
Combining (\ref{id.3}) with (\ref{id.2}) gives
\begin{multline} \label{id.4}
   \left\|\frac{(\log|x'|)^k(x'\cdot \nabla_N)^k (\log|x'|(x'\cdot \nabla_N) f)}{|x'|^\frac{N}{2}}\right\|^2_{L^2(\mathbb{R}^n)}=\\=\left\|\frac{(\log|x'|)^{k+1}(x'\cdot \nabla_N)^{k+1}f}{|x'|^\frac{N}{2}}\right\|^2_{L^2(\mathbb{R}^n)}-(k^2+k)\left\| \frac{(\log|x'|)^k(x'\cdot\nabla_N)^{k}f}{|x'|^\frac{N}{2}}\right\|^2_{L^2(\mathbb{R}^n)}.
\end{multline}
Now by combining (\ref{id.4}), (\ref{id.1}) and \eqref{A.m.m+1} we get 
\begin{multline} \label{id.5}
    4^k \left\|\frac{(\log|x'|)^{k+1}(x'\cdot \nabla_N)^{k+1}f}{|x'|^\frac{N}{2}}\right\|^2_{L^2(\mathbb{R}^n)}=k(k+1) 4^{k}\left\|\frac{\log^{k}|x'|(x'\cdot \nabla_N)^{k}f}{|x'|^\frac{N}{2}}\right\|^2_{L^2(\mathbb{R}^n)} \\+ a_k \left\|\frac{\log|x'|(x'\cdot \nabla_N) f}{|x'|^\frac{N}{2}}\right\|^2_{L^2(\mathbb{R}^n)}+ \sum_{m=1}^{k}O(k,m)A^{m,m+1}.
\end{multline}
Using the identities (\ref{higherorder}) and (\ref{rem. of crit. sob. L2}) on the right-hand side of (\ref{id.5}) we get
\begin{multline} \label{id.6} 
    4^k \left\|\frac{(\log|x'|)^{k+1}(x'\cdot \nabla_N)^{k+1}f}{|x'|^\frac{N}{2}}\right\|^2_{L^2(\mathbb{R}^n)}\\=k(k+1) a_k \left\|\frac{f}{|x'|^\frac{N}{2}}\right\|^2_{L^2(\mathbb{R}^n)} +k(k+1) \sum_{m=1}^{k}O(k,m)A^{m-1,m} \\ + \frac{1}{4}a_k \left\|\frac{f}{|x'|^\frac{N}{2}}\right\|^2_{L^2(\mathbb{R}^n)} +\frac{1}{4}a_k A^{0,1}+ \sum_{m=1}^{k}O(k,m)A^{m,m+1}. 
\end{multline}
By multiplying both sides of (\ref{id.6}) by 4, then using $$4k(k+1) \sum_{m=1}^{k}O(k,m)A^{m-1,m}=4k(k+1)O(k,1)A^{0,1}+4k(k+1) \sum_{m=2}^{k}O(k,m)A^{m-1,m}$$
and $O(k,1)=a_{k},$ from (\ref{coefficient}) we get
\begin{multline} \label{id.7}
    4^{k+1} \left\|\frac{(\log|x'|)^{k+1}(x'\cdot \nabla_N)^{k+1}f}{|x'|^\frac{N}{2}}\right\|^2_{L^2(\mathbb{R}^n)}=\big(4k(k+1)+1\big) a_k \left\|\frac{f}{|x'|^\frac{N}{2}}\right\|^2_{L^2(\mathbb{R}^n)} \\ +\big(4k(k+1)+1\big)a_kA^{0,1}+4k(k+1) \sum_{m=2}^{k}O(k,m)A^{m-1,m}+ 4\sum_{m=1}^{k}O(k,m)A^{m,m+1}\\=a_{k+1} \left\|\frac{f}{|x'|^\frac{N}{2}}\right\|^2_{L^2(\mathbb{R}^n)}+O(k+1,1)A^{0,1}\\+ \sum_{m=1}^{k-1}\big[4k(k+1)O(k,m+1)+4O(k,m)\big]A^{m,m+1}+4O(k,k)A^{k,k+1},
\end{multline} where in the last line we have used $(4k(k+1)+1)a_k=a_{k+1}$ since (\ref{sq.double}) and $a_{k+1}=O(k+1,1)$ from (\ref{coefficient}).

Rewriting the sum in the last line of  (\ref{id.7}) by replacing $m$ with $m-1$ and $4O(k,k)$ with $O(k+1,k+1)$, since $O(k,k)=4^{k-1}$ in \eqref{coefficient}, then using Lemma \ref{lemma 2}, we get 
\begin{multline} \label{id.11}
4^{k+1} \left\|\frac{(\log|x'|)^{k+1}(x'\cdot \nabla_N)^{k+1}f}{|x'|^\frac{N}{2}}\right\|^2_{L^2(\mathbb{R}^n)}=a_{k+1} \left\|\frac{f}{|x'|^\frac{N}{2}}\right\|^2_{L^2(\mathbb{R}^n)}+O(k+1,1)A^{0,1}\\+\sum_{m=2}^{k}\big[4k(k+1)O(k,m)+4O(k,m-1)\big] A^{m-1,m}+O(k+1,k+1)A^{k,k+1}
    \\=a_{k+1}  \left\|\frac{f}{|x'|^\frac{N}{2}}\right\|^2_{L^2(\mathbb{R}^n)}+O(k+1,1)A^{0,1}+\sum_{m=2}^{k}O(k+1,m) A^{m-1,m}+O(k+1,k+1)A^{k,k+1}\\=a_{k+1}  \left\|\frac{f}{|x'|^\frac{N}{2}}\right\|^2_{L^2(\mathbb{R}^n)}+\sum_{m=1}^{k+1}O(k+1,m) A^{m-1,m}.
\end{multline}
Thus, we have checked the $(k+1)^{th}$ order of \eqref{higherorder} using the $k^{th}$ order. 
Then, Theorem \ref{th4.1} is true for all $k$.
\end{proof}

\begin{proof}[Proof of Corollary \ref{higher order inequality}] 
To derive the higher‐order inequality in \eqref{higherorder inequality}, we start by applying the substitution 
\[
f = \log|x'|(x'\cdot\nabla_N)f
\]
into the Sobolev type inequality
\begin{equation}\label{1st}
  \frac{1}{2^{2}}\left\|\frac{f}{|x'|^{\frac{N}{2}}}\right\|_{L^2(\mathbb{R}^n)}^2 \leq \left\|\frac{\log|x'|(x'\cdot \nabla_N) f}{|x'|^{\frac{N}{2}}}\right\|_{L^2(\mathbb{R}^n)}^2,
\end{equation}
which is given in \eqref{sob}. This substitution leads to
\[
\frac{1}{2^{2}}\left\|\frac{\log|x'|(x'\cdot\nabla_N)f}{|x'|^{\frac{N}{2}}}\right\|_{L^2(\mathbb{R}^n)}^2 \leq \left\|\frac{\log|x'|(x'\cdot \nabla_N)\log|x'|(x'\cdot\nabla_N)f}{|x'|^{\frac{N}{2}}}\right\|_{L^2(\mathbb{R}^n)}^2.
\]
Using the identity \eqref{id.4} for \(k=1\),
\begin{multline}\label{kth}
   \left\|\frac{(\log|x'|)^k (x'\cdot \nabla_N)^k \Bigl(\log|x'|(x'\cdot \nabla_N) f\Bigr)}{|x'|^{\frac{N}{2}}}\right\|_{L^2(\mathbb{R}^n)}^2\\
   = \left\|\frac{(\log|x'|)^{k+1}(x'\cdot \nabla_N)^{k+1}f}{|x'|^{\frac{N}{2}}}\right\|_{L^2(\mathbb{R}^n)}^2 - (k^2+k)\left\|\frac{(\log|x'|)^k (x'\cdot\nabla_N)^{k}f}{|x'|^{\frac{N}{2}}}\right\|_{L^2(\mathbb{R}^n)}^2,
\end{multline}
we obtain, after simplification, the second-order inequality
\begin{equation}\label{2nd}
  \frac{3^2}{2^{4}}\left\|\frac{f}{|x'|^{\frac{N}{2}}}\right\|_{L^2(\mathbb{R}^n)}^2 \leq \left\|\frac{(\log|x'|)^2 (x'\cdot\nabla_N)^2 f}{|x'|^{\frac{N}{2}}}\right\|_{L^2(\mathbb{R}^n)}^2.
\end{equation}
Next, by applying the substitution \(f = \log|x'|(x'\cdot\nabla_N)f\) in \eqref{2nd} and employing \eqref{kth} together with \eqref{2nd} and \eqref{1st}, we obtain the third-order version:
\begin{equation}\label{3rd}
  \frac{3^2 5^2}{2^6}\left\|\frac{f}{|x'|^{\frac{N}{2}}}\right\|_{L^2(\mathbb{R}^n)}^2 \leq \left\|\frac{(\log|x'|)^3 (x'\cdot\nabla_N)^3 f}{|x'|^{\frac{N}{2}}}\right\|_{L^2(\mathbb{R}^n)}^2.
\end{equation}
By repeating this iterative procedure, we arrive at
\begin{equation}\label{higherorder inequality 0}
  \frac{(2k-1)!!^2}{2^{2k}}\left\|\frac{f}{|x'|^{\frac{N}{2}}}\right\|_{L^2(\mathbb{R}^n)}^2 \leq \left\|\frac{(\log|x'|)^k (x'\cdot \nabla_N)^k f}{|x'|^{\frac{N}{2}}}\right\|_{L^2(\mathbb{R}^n)}^2,
\end{equation}
which implies the higher-order inequality stated in \eqref{higherorder inequality}. Note that this inequality can also be derived by dropping the positive remainder terms on the right-hand side of \eqref{higherorder}.

To show the sharpness of the constant $\frac{(2k-1)!!^2}{2^{2k}}$ in \eqref{higherorder inequality 0}, we consider the test function 
$$g_A(x)=\phi(x'')\psi_\delta(x')(\log|x'|)^A,$$ 
for $A<-\frac{1}{2}$ and small $\delta>0$, where $\phi(x'')\in C_0^\infty(\mathbb{R}^{n-N})$ is positive function. Here, as in \cite[Proof of Theorem 3.1, page 665]{San22}, let us choose $\psi_\delta$ to be smooth, radially symmetric function with the properties
\begin{equation}
    \psi_\delta(x')\equiv\begin{cases}0, \       \ 0\leq|x'|\leq1+\delta \\ 0\leq\psi_\delta\leq1, \ \ 1+\delta\leq|x'|\leq1+2\delta \\ 
    1, \ \ |x'|\geq1+2\delta,
    \end{cases}
\end{equation}
and $\left|\big(\frac{d}{d|x'|}\big)^{k}\psi_\delta(x')\right|\leq \frac{C_k}{\delta^k}$ for each $k$.

We will also use $\tilde{B}_R:=\{x'\in\mathbb{R}^N : |x'|\leq R\}$.
Taking into account $$|(x'\cdot\nabla_N)^k(\log|x'|)^A|=\prod_{j=0}^k|A-j|(\log|x'|)^{A-k},$$
we have 
\begin{align*}
\label{higherorder inequality 1}
  \frac{(2k-1)!!^2}{2^{2k}}&\leq\frac{\int_{\mathbb{R}^n}\frac{|(\log|x'|)^k(x'\cdot \nabla_N)^k g_A|^2}{|x'|^N}dx}{\int_{\mathbb{R}^n}\frac{|g_A|^2}{|x'|^N}dx}\\&\leq  \frac{\int_{\mathbb{R}^{n-N}}\int_{\mathbb{R}^N\backslash \tilde{B}_{1+2\delta}}\frac{(\log|x'|)^{2k}|(x'\cdot \nabla_N)^k g_A|^2}{|x'|^N}dx'dx''}{\int_{\mathbb{R}^{n-N}}\int_{\mathbb{R}^N\backslash \tilde{B}_{1+2\delta}}\frac{|g_A|^2}{|x'|^N}dx'dx''}\\&+ \frac{\int_{\mathbb{R}^{n-N}}\int_{\tilde{B}_{1+2\delta}\backslash \tilde{B}_{1+\delta}}\frac{(\log|x'|)^{2k}|(x'\cdot \nabla_N)^k g_A|^2}{|x'|^N}dx'dx''}{\int_{\mathbb{R}^{n-N}}\int_{\mathbb{R}^N\backslash \tilde{B}_{1+2\delta}}\frac{|g_A|^2}{|x'|^N}dx'dx''}\\&= \frac{\int_{\mathbb{R}^{n-N}}\int_{\mathbb{R}^N\backslash \tilde{B}_{1+2\delta}}\frac{|\phi(x'')|^2|(\log|x'|)^k(x'\cdot \nabla_N)^k (\log|x'|)^A|^2}{|x'|^N}dx'dx''}{\int_{\mathbb{R}^{n-N}}\int_{\mathbb{R}^N\backslash \tilde{B}_{1+2\delta}}\frac{|\phi(x'')|^2|(\log|x'|)^A|^2}{|x'|^N}dx'dx''}\\&+\frac{\int_{\mathbb{R}^{n-N}}\int_{\tilde{B}_{1+2\delta}\backslash \tilde{B}_{1+\delta}}\frac{|\phi(x'')|^2|(x'\cdot \nabla_N)^k (\psi_\delta(x')(\log|x'|)^A)|^2}{|x'|^N|(\log|x'|)^{-2k}}dx'dx''}{\int_{\mathbb{R}^{n-N}}\int_{\mathbb{R}^N\backslash \tilde{B}_{1+2\delta}}\frac{|\phi(x'')|^2|(\log|x'|)^A|^2}{|x'|^N}dx'dx''}\\&\leq  \frac{\prod_{j=0}^{k-1}(A-j)^2\int_{\mathbb{R}^{n-N}}\int_{\mathbb{R}^N\backslash \tilde{B}_{1+2\delta}}\frac{|\phi(x'')|^2|(\log|x'|)^k(\log|x'|)^{A-k}|^2}{|x'|^N}dx'dx''}{\int_{\mathbb{R}^{n-N}}\int_{\mathbb{R}^N\backslash \tilde{B}_{1+2\delta}}\frac{|\phi(x'')|^2|(\log|x'|)^A|^2}{|x'|^N}dx'dx''}\\&+\frac{\int_{\mathbb{R}^{n-N}}|\phi(x'')|^2\int_{\tilde{B}_{1+2\delta}\backslash \tilde{B}_{1+\delta}}\frac{|(|x'|\frac{d}{d|x'|})^k (\psi_\delta(x')(\log|x'|)^A)|^2}{|x'|^N|(\log|x'|)^{-2k}}dx'dx''}{\int_{\mathbb{R}^{n-N}}\int_{\mathbb{R}^N\backslash \tilde{B}_{1+2\delta}}\frac{|\phi(x'')|^2|(\log|x'|)^A|^2}{|x'|^N}dx'dx''}\\&\leq  \frac{(2k-1)!!^2}{2^{2k}}+o(1), \ \ \left(A\nearrow -\frac{1}{2}\right)
\end{align*}
where we have used 
$$\frac{\int_{\mathbb{R}^{n-N}}|\phi(x'')|^2\int_{\tilde{B}_{1+2\delta}\backslash \tilde{B}_{1+\delta}}\frac{|(|x'|\frac{d}{d|x'|})^k (\psi_\delta(x')(\log|x'|)^A)|^2}{|x'|^N(\log|x'|)^{-2k}}dx'dx''}{\int_{\mathbb{R}^{n-N}}\int_{\mathbb{R}^N\backslash \tilde{B}_{1+2\delta}}\frac{|\phi(x'')|^2|(\log|x'|)^A|^2}{|x'|^N}dx'dx''}
=o(1),$$
since the integral in the denominator diverges as $A \nearrow -\frac{1}{2},$ while the numerator converges because the integrand is bounded on the annulus 
\(\tilde{B}_{1+2\delta}\setminus \tilde{B}_{1+\delta}\) due to the estimate
\begin{multline*}
    \left|\left(|x'|\frac{d}{d|x'|}\right)^k\left(\psi_\delta(x')(\log|x'|)^A\right)\right|\leq\sum_{i=1}^{k}S(k,i)|x'|^i
\left|\left(\frac{d}{d|x'|}\right)^i\left( \psi_\delta(x')(\log|x'|)^A\right)\right|\\=\sum_{i=1}^{k}S(k,i)|x'|^i
\sum_{j=0}^{i}\binom{i}{j}\left|\left(\frac{d}{d|x'|}\right)^{i-j}\psi_\delta(x')\right|\left(\frac{d}{d|x'|}\right)^{j}(\log|x'|)^A\\\leq\sum_{i=1}^{k}S(k,i)|x'|^i
\sum_{j=0}^{i}\binom{i}{j}\frac{C_{i-j}}{\delta^{i-j}}\prod_{l}^{j-1}(A-l)(\log|x'|)^{A-j}.
\end{multline*} Here, we use formula \eqref{Eul.op.Stirling} for expansion of $\left(|x'|\frac{d}{d|x'|}\right)^k$, general Leibniz rule, and then $\left|\big(\frac{d}{d|x'|}\big)^{k}\psi_\delta(x')\right|\leq \frac{C_k}{\delta^k}$. Thus, the sharpness of the constant in \eqref{higherorder inequality} is established. In fact, the sharp constant would be attained if and only if all the nonnegative remainder terms in the identity \eqref{higherorder} vanish. However, as shown in Theorem \ref{th2.1}, the remainder term for \(m=1\) in \eqref{higherorder} vanishes if and only if 
\[
f = (\log|x'|)^{-\frac{1}{2}}\varphi\Bigl(\frac{x'}{|x'|},x''\Bigr).
\]
This condition forces the left-hand side of \eqref{higherorder inequality} to become infinite unless \(\varphi = f = 0\). Consequently, the sharp constant in \eqref{higherorder inequality} is not attained.
\end{proof}

We can derive higher-order identities and inequalities for stratified and homogeneous Lie groups by applying the same method. Let us present these results in the following theorems and corollaries. 
\begin{thm} \label{higher order on strtfd}
    Let $\mathbb{G}$ be a stratified group with N being the dimension of the first stratum and $k\in \mathbb{N}.$  Then for any $f \in C^\infty_0(\mathbb{G}\backslash\{{x'=0}\})$, we have 
\begin{multline*}
    \frac{4^k}{a_k} \left\|\frac{(\log|x'|)^k(x'\cdot \nabla_H)^k f}{|x'|^\frac{N}{2}}\right\|^2_{L^2(\mathbb{G})}=  \left\|\frac{f}{|x'|^\frac{N}{2}}\right\|^2_{L^2(\mathbb{G})} \\+ \sum_{m=1}^{k}\frac{O(k,m)}{a_k}\Bigg\|\sum_{l=0}^{m-1}S(m-1,l)\frac{(\log|x'|)^l(x'\cdot \nabla_H)^l f}{|x'|^\frac{N}{2}} \\ +2\sum_{\kappa=1}^{m}S(m,\kappa)\frac{(\log|x'|)^\kappa(x'\cdot \nabla_H)^\kappa f}{|x'|^\frac{N}{2}}\Bigg\|^2_{L^2(\mathbb{G})},
\end{multline*}
where $a_k$ is in (\ref{sq.double}), $O(k, m)$ is in (\ref{coefficient}), and $S(m, \kappa)$ and $S(m-1,l)$ are in (\ref{stirling}).
\end{thm}

\begin{cor} \label{higher order inequality on strtfd}
     Let $\mathbb{G}$ be a stratified group with N being the dimension of the first stratum and $k\in \mathbb{N}.$  Then for any $f \in C^\infty_0(\mathbb{G}\backslash\{{x'=0}\})$, we have 
\begin{equation} \label{high.or ineq. on strtfd}
  \left\|\frac{f}{|x'|^\frac{N}{2}}\right\|_{L^2(\mathbb{G})}\leq \frac{2^k}{a(k)} \left\|\frac{(\log|x'|)^k(x'\cdot \nabla_H)^k f}{|x'|^\frac{N}{2}}\right\|_{L^2(\mathbb{G})},
\end{equation}
where $a(k)=(2k-1)!!$ is the double factorial of an odd number.
\end{cor}

\begin{thm} \label{higher order on hom}
    Let $\mathbb{G}$ be a homogeneous group of homogeneous dimension $Q$, and let $|\cdot|$ be any homogeneous quasi-norm on $\mathbb{G}$. Let $\mathcal{R}_{|x|}=\frac{d}{d|x|}$ be the radial derivative. Then for any $f \in C^\infty_0(\mathbb{G}\backslash\{{0}\})$, we have 
\begin{multline} \label{higherorder on hom}
    \frac{4^k}{a_k} \left\|\frac{(\log|x|)^k(|x|\mathcal{R}_{|x|} )^k f}{|x|^{\frac{Q}{2}}}\right\|^2_{L^2(\mathbb{G})}=  \left\|\frac{f}{|x|^\frac{Q}{2}}\right\|^2_{L^2(\mathbb{G})} \\+ \sum_{m=1}^{k}\frac{O(k,m)}{a_k}\Bigg\|\sum_{l=0}^{m-1}S(m-1,l)\frac{(\log|x|)^l(|x|\mathcal{R}_{|x|})^l f}{|x|^{\frac{Q}{2}}} \\+2\sum_{\kappa=1}^{m}S(m,\kappa)\frac{(\log|x|)^\kappa(|x|\mathcal{R}_{|x|})^\kappa f}{|x|^{\frac{Q}{2}}}\Bigg\|^2_{L^2(\mathbb{G})}, 
\end{multline}
where $a_k$ is in (\ref{sq.double}), $O(k, m)$ is in (\ref{coefficient}), $S(m, \kappa)$ and $S(m-1,l)$ are in (\ref{stirling}).
\end{thm}
\begin{rem}
   The identity \eqref{higherorder on hom} is the critical case $\alpha=\frac{Q}{2}$ of the identity in \cite[Theorem 3.10]{RSY18b}.
\end{rem}
\begin{cor} \label{higher order inequality on hom}
     Let $\mathbb{G}$ be a homogeneous group of homogeneous dimension $Q$, and let $|\cdot|$ be any homogeneous quasi-norm on $\mathbb{G}$. Let $\mathcal{R}_{|x|}=\frac{d}{d|x|}$ be the radial derivative. Then for any $f \in C^\infty_0(\mathbb{G}\backslash\{{0}\})$, we have 
\begin{equation} \label{high.or ineq. on hom}
  \left\|\frac{f}{|x|^\frac{Q}{2}}\right\|_{L^2(\mathbb{G})}\leq \frac{2^k}{a(k)} \left\|\frac{(\log|x|)^k(|x|\mathcal{R}_{|x|})^k f}{|x|^{\frac{Q}{2}}}\right\|_{L^2(\mathbb{G})},
\end{equation}
where $a(k)=(2k-1)!!$ is the double factorial of odd number, and $\frac{2^k}{a(k)}$ is the sharp constant and non-attainable.
\end{cor}
\begin{rem}
   The inequality \eqref{high.or ineq. on hom} is the critical case $\alpha=\frac{Q}{2}$ of the inequality in \cite[Theorem 3.10]{RSY18b}.
\end{rem}
\section{Applications} 
In this section we investigate (critical) Caffarelli-Kohn-Nirenberg type inequalities and (critical) uncertainty type principles with logarithmic weights. 

\begin{thm} \label{Lp caffarelli with remainder}
Let $x=(x',x'') \in \mathbb{R}^N\times\mathbb{R}^{n-N}$. Let $1<p,q<\infty$, $0<r<\infty$, with $p+q\geq r$, $\delta \in [0, 1] \cap \left[\frac{r-q}{r},\frac{p}{r}\right]$ and $b,c \in \mathbb{R}$. Assume that 
$$\frac{\delta r}{p}+\frac{(1-\delta)r}{q}=1 \ \ and  \ \ c=-\frac{N}{p}\delta+b(1-\delta).$$
Then we have the following Caffarelli-Kohn-Nirenberg type inequality with a logarithmic weight for any complex-valued function $f\in C^\infty_0(\mathbb{R}^n \backslash \{{x'=0}\})$:
\small    
\begin{multline} \label{Lp c.caf with remainder}
\||x'|^c f\|_{L^r(\mathbb{R}^n)}\leq p^\delta \Bigg(\left\|\frac{\log|x'|(x'\cdot \nabla_N) f}{|x'|^\frac{N}{p}} \right\|_{L^p(\mathbb{R}^n)}^p-\\
    -\frac{1}{p^p}\int_{\mathbb{R}^n}C_p\left(\frac{p\log|x'|(x'\cdot \nabla_N) f}{|x'|^\frac{N}{p}} ,\frac{p(\log|x'|)^{-\frac{1}{p}+1}}{|x'|^\frac{N}{p}}(x'\cdot\nabla_N)\left(f(\log|x'|)^\frac{1}{p}\right)\right)dx\Bigg)^\frac{\delta}{p}
    \\\times\||x'|^bf\|^{1-\delta}_{L^q(\mathbb{R}^n)},
\end{multline} 
\normalsize
where $|x'|$  is the Euclidean norm on $\mathbb{R}^N$, $\nabla_N$ is the standard gradient on $\mathbb{R}^N$ and $C_p$ is functional from $\eqref{Cp}$. The constant in the inequality \eqref{Lp c.caf with remainder} is sharp for $p=q$ with $\frac{N}{p}=-b$ or for $\delta=0,1.$
\end{thm}

\begin{rem}
    The inequality \eqref{Lp c.caf with remainder} provides the critical case of \cite[Theorem 8]{KY24}.
\end{rem}

\begin{proof}[Proof of Theorem \ref{Lp caffarelli with remainder}.]
Case $\delta=0$. In this case, we 
          have $q=r$ and $b=c$ by $\frac{\delta r}{2}+\frac{(1-\delta)r}{q}=1$ and $c=-\frac{N}{2}\delta+b(1-\delta),$ respectively. Then, the inequality (\ref{Lp c.caf with remainder}) reduces to the trivial estimate
$$\||x'|^b f\|_{L^q(\mathbb{R}^n)}\leq \|x'|^bf\|_{L^q(\mathbb{R}^n)}.$$
Case $\delta=1$. Notice that in this case, $r=p$ and $c=-\frac{N}{p}$. By \eqref{Cp rem. of crit. sob.}, we have the equality
\begin{multline}
\left\|\frac{f}{|x'|^\frac{N}{p}}\right\|_{L^p(\mathbb{R}^n)}= p\Bigg(\left\|\frac{\log|x'|(x'\cdot \nabla_N) f}{|x'|^\frac{N}{p}} \right\|_{L^p(\mathbb{R}^n)}^p\\-\frac{1}{p^p}\int_{\mathbb{R}^n}C_p\left(\frac{p\log|x'|(x'\cdot \nabla_N) f}{|x'|^\frac{N}{p}} ,\frac{p(\log|x'|)^{-\frac{1}{p}+1}}{|x'|^\frac{N}{p}}(x'\cdot\nabla_N)\left(f(\log|x'|)^\frac{1}{p}\right)\right)dx\Bigg)^\frac{1}{p}.
\end{multline}
Case $\delta\in(0,1)\cap \left [\frac{r-q}{r},\frac{p}{r}\right] $. Taking into account $c=-\frac{N}{p}\delta+b(1-\delta)$, a direct calculation gives 
$$\||x'|^c f\|_{L^r(\mathbb{R}^n)} = \left( \int_{\mathbb{R}^n}|x'|^{cr}|f|^rdx \right )^\frac{1}{r} = \left( \int_{\mathbb{R}^n}\frac{|f|^{\delta r}}{|x'|^{\frac{N}{p}\delta r}}\frac{|f|^{(1-\delta)r}}{|x'|^{-b(1-\delta)r}}dx \right)^\frac{1}{r}.$$
Since we have $\delta\in(0,1)\cap \left [\frac{r-q}{r},\frac{p}{r}\right]$ and $q+p\geq r$, then by using H\"older's inequality for $\frac{\delta r}{p}+\frac{(1-\delta)r}{q}=1$, we obtain 
\begin{equation} \label{holder.r,p,q}
\begin{split}
\||x'|^c f\|_{L^r(\mathbb{R}^n)}&\leq \left(\int_{\mathbb{R}^n}\frac{|f|^p}{|x'|^N}dx\right)^\frac{\delta}{p}\left(\int_{\mathbb{R}^n}\frac{|f|^q}{|x'|^{-bq}}dx\right)^\frac{1-\delta}{q} \\ &= \Bigg \|\frac{f}{|x'|^\frac{N}{p}}\Bigg \|^\delta_{L^p}\Bigg \|\frac{f}{|x'|^{-b}}\Bigg \|^{1-\delta}_{L^q}.
\end{split}
\end{equation}

Here we note that when $q=p$ and $\frac{N}{p}=-b$, H\"older's equality condition is satisfied for any function. 

By (\ref{Cp rem. of crit. sob.}) we have the identity
\begin{multline} \label{rem. of crit. sob. ^delta}
\left\|\frac{f}{|x'|^\frac{N}{p}}\right\|_{L^p(\mathbb{R}^n)}^\delta = p^\delta \Bigg(\left\|\frac{\log|x'|(x'\cdot \nabla_N) f}{|x'|^\frac{N}{p}} \right\|_{L^p(\mathbb{R}^n)}^p\\-\frac{1}{p^p}\int_{\mathbb{R}^n}C_p\left(\frac{p\log|x'|(x'\cdot \nabla_N) f}{|x'|^\frac{N}{p}} ,\frac{p(\log|x'|)^{-\frac{1}{p}+1}}{|x'|^\frac{N}{p}}(x'\cdot\nabla_N)\left(f(\log|x'|)^\frac{1}{p}\right)\right)dx\Bigg)^\frac{\delta}{p}.
\end{multline}

Putting this in \eqref{holder.r,p,q}, one has
\small
\begin{multline}
\||x'|^c f\|_{L^r(\mathbb{R}^n)}\leq p^\delta \Bigg(\left\|\frac{\log|x'|(x'\cdot \nabla_N) f}{|x'|^\frac{N}{p}} \right\|_{L^p(\mathbb{R}^n)}^p\\-\frac{1}{p^p}\int_{\mathbb{R}^n}C_p\left(\frac{p\log|x'|(x'\cdot \nabla_N) f}{|x'|^\frac{N}{p}} ,\frac{p(\log|x'|)^{-\frac{1}{p}+1}}{|x'|^\frac{N}{p}}(x'\cdot\nabla_N)\left(f(\log|x'|)^\frac{1}{p}\right)\right)dx\Bigg)^\frac{\delta}{p}\\\times\Bigg \|\frac{f}{|x'|^{-b}}\Bigg \|^{1-\delta}_{L^q}.
\end{multline}
\normalsize

The constant in \eqref{rem. of crit. sob. ^delta} is sharp, from \eqref{Cp rem. of crit. sob.}. In the case $q=p$ and $\frac{N}{p}=-b$ H\"older's equality condition for the inequality \eqref{holder.r,p,q} holds true. Therefore, the constant in \eqref{Lp c.caf with remainder} is sharp when $q=p$ and $\frac{N}{p}=-b$.
\end{proof}

\begin{cor}[Uncertainty type principle with a logarithmic weight] \label{uncertainty}
    Let $x=(x',x'') \in \mathbb{R}^N\times\mathbb{R}^{n-N}, 1< N\leq n$.  Then for any $f \in C^\infty_0(\mathbb{R}^n\backslash\{{x'=0}\})$ and $\frac{1}{N}+\frac{1}{q}=1,$ we have
    \begin{equation} \label{crit. uncertainty} 
     \int_{\mathbb{R}^n}|f|^2dx\leq N\left\||x'|^{-1}\log|x'|(x'\cdot\nabla_N) f\right\|_{L^N(\mathbb{R}^n)} \left\||x'|f\right\|_{L^q(\mathbb{R}^n)},
    \end{equation}
where $\nabla_N$ is the standard gradient on $\mathbb{R}^N$.
\end{cor}
\begin{proof}[Proof of Corollary \ref{uncertainty}]
In \eqref{Lp c.caf with remainder}, we drop the remainder term and take $r=2$, $c=0$, $p=N$, and $b=1$, then $\delta=\frac{1}{2}$. Therefore, we get  
    \begin{equation*}   
    \|f\|_{L^2(\mathbb{R}^n)}\leq N^\frac{1}{2} \Bigg\|\frac{(x'\cdot\nabla_N)f}{|x'|}\log|x'|  \Bigg\|^{\frac{1}{2}}_{L^N(\mathbb{R}^n)}\||x'|f\|^\frac{1}{2}_{L^q(\mathbb{R}^n)},
    \end{equation*}
which implies \eqref{crit. uncertainty}.
\end{proof}
\begin{rem}
In case $N=q=n=2$, the inequality \eqref{crit. uncertainty} reduces to 
\begin{equation} \label{critical case of HPW}     \left(\int_{\mathbb{R}^2}|f|^2dx\right)^2 \leq 4 \int_{\mathbb{R}^2} (\log|x|)^2\left|\frac{x}{|x|}\cdot\nabla f\right|^2dx\int_{\mathbb{R}^2}|x|^2 |f|^2dx,
\end{equation}
where $|x|$ is the standard Euclidean norm on $\mathbb{R}^2$ and $\nabla$ is the standard gradient on $\mathbb{R}^2$. \\

The inequality \eqref{critical case of HPW} is the critical case of the Heisenberg-Pauli-Weyl uncertainty type principle from \cite[Remark 2.10]{RS17a}.
By using the Schwarz inequality on the right-hand side of \eqref{critical case of HPW} we get
\begin{equation}
\left(\int_{\mathbb{R}^2}|f|^2dx\right)^2 \leq 4 \int_{\mathbb{R}^2} (\log|x|)^2\left|\nabla f\right|^2dx\int_{\mathbb{R}^2}|x|^2 |f|^2dx,
\end{equation}
which is the critical case of the classical Heisenberg-Pauli-Weyl uncertainty principle. 
\end{rem}
\begin{rem}
In the case where $q=1, \ r=2, \ N=p=2^*=\frac{2n}{n-2}$, and $c=b=-1$, the inequality \eqref{Lp caffarelli with remainder} without the remainder term reduces to the following Nash-type inequality:
    \begin{equation}\label{Nash type1}
\left\|\frac{f}{|x'|}\right\|_{L^2(\mathbb{R}^n)}^{1+\frac{2}{n}} \leq {\frac{2n}{n-2}}\left\|\frac{f}{|x'|}\right\|^\frac{2}{n}_{L^1(\mathbb{R}^n)} \left\|\frac{\log|x'|(x'\cdot \nabla_N) f}{|x'|} \right\|_{L^{2^*}(\mathbb{R}^n)}.
\end{equation}
\end{rem}

By substituting inequality (\ref{higherorder inequality}) into $p=2$ version of \eqref{holder.r,p,q} we obtain the higher order Caffarelli-Kohn-Nirenberg type inequality with a logarithmic weight.

\begin{cor}
Let $x=(x',x'') \in \mathbb{R}^N\times\mathbb{R}^{n-N}$. Let $1<q<\infty$, $0<r<\infty$ with $q+2\geq r$, $\delta \in [0, 1] \cap \left[\frac{r-q}{r},\frac{2}{r}\right]$ and $b,c \in \mathbb{R}$. Assume that 
$$\frac{\delta r}{2}+\frac{(1-\delta)r}{q}=1 \ \ and  \ \ c=-\frac{N}{2}\delta+b(1-\delta).$$
Then we have the following higher order critical Caffarelli-Kohn-Nirenberg type inequality for all $f\in C^\infty_0(\mathbb{R}^n \backslash \{{x'=0}\})$ and any $k\in\mathbb{N}$:
\begin{equation} \label{h.c.caf}
\||x'|^c f\|_{L^r(\mathbb{R}^n)}\leq \bigg(\frac{2^k}{a(k)}\bigg)^\delta \||x'|^\frac{-N}{2}(\log|x'|)^k(x'\cdot \nabla_N)^k f\|_{L^2(\mathbb{R}^n)}^\delta \||x'|^bf\|^{1-\delta}_{L^q(\mathbb{R}^n)},
\end{equation} 
where $|x'|$ is the Euclidean norm on $\mathbb{R}^N$, $\nabla_N$ is the standard gradient on $\mathbb{R}^N$, and $a(k)=(2k-1)!!$ is the double factorial of odd number.
\end{cor}
\begin{rem}
In case when $N=2k$, $b=k$, $c=0$, $r=q=2$  and $\delta=\frac{1}{2}$, then \eqref{h.c.caf} implies the following higher order uncertainty type principle with a logarithmic weight:
\begin{equation} \label{h.log.uncer}
\|f\|_{L^2(\mathbb{R}^n)}^2\leq \frac{2^k}{a(k)} \||x'|^{-k}(\log|x'|)^k(x'\cdot \nabla_{2k})^k f\|_{L^2(\mathbb{R}^n)}\||x'|^kf\|_{L^2(\mathbb{R}^n)},
\end{equation} 
where $|x'|$ is the Euclidean norm on $\mathbb{R}^{2k}$, $\nabla_{2k}$ is the standard gradient on $\mathbb{R}^{2k}$, and $a(k)=(2k-1)!!$ is the double factorial of odd number.
\end{rem}
To obtain the aforementioned results on the stratified Lie groups, we employ the same method. For the sake of clarity and brevity, we will now present these results without providing a formal proof.
\vspace{4pt}

For any complex-valued function, by using identity \eqref{rem. of crit. sob. on strtfd} we obtain the following Caffarelli-Kohn-Nirenberg type inequality with a logarithmic weight and a remainder term on stratified Lie groups.

\begin{thm} \label{Lp caffarelli with remainder on strtfd}
Let $\mathbb{G}$ be a stratified group with N being the dimension of the first stratum. Let $1<p,q<\infty$, $0<r<\infty$, with $p+q\geq r$, $\delta \in [0, 1] \cap \left[\frac{r-q}{r},\frac{p}{r}\right]$ and $b,c \in \mathbb{R}$. Assume that 
$$\frac{\delta r}{p}+\frac{(1-\delta)r}{q}=1 \ \ and  \ \ c=-\frac{N}{p}\delta+b(1-\delta).$$
Then we have the following Caffarelli-Kohn-Nirenberg type inequality with a logarithmic weight for any complex-valued function $f\in C^\infty_0(\mathbb{G} \backslash \{{x'=0}\})$:
\small    
\begin{multline} \label{Lp c.caf with remainder on strtfd}
\||x'|^c f\|_{L^r(\mathbb{G})}\leq p^\delta \Bigg(\Bigg\|\frac{\log|x'|(x'\cdot \nabla_H) f}{|x'|^\frac{N}{p}}\Bigg\|_{L^p(\mathbb{G})}^p\\-\frac{1}{p^p}\int_\mathbb{G}C_p\left(\frac{p\log|x'|(x'\cdot \nabla_H) f}{|x'|^\frac{N}{p}} ,\frac{p(\log|x'|)^{-\frac{1}{p}+1}}{|x'|^\frac{N}{p}}(x'\cdot\nabla_H)\left(f(\log|x'|)^\frac{1}{p}\right)\right)dx\Bigg)^\frac{\delta}{p}\\\times
    \||x'|^bf\|^{1-\delta}_{L^q(\mathbb{G})},
\end{multline} 
\normalsize
where $|x'|$  is the Euclidean norm on $\mathbb{R}^N$, $\nabla_H$ is the horizontal gradient on $\mathbb{G}$. The constant in the inequality \eqref{Lp c.caf with remainder on strtfd} is sharp for $p=q$ with $\frac{N}{p}=-b$ or for $\delta=0,1.$
\end{thm}
\begin{rem}
By dropping the remainder term on the right-hand side of the inequality \eqref{Lp c.caf with remainder on strtfd}, we obtain a Caffarelli-Kohn-Nirenberg type inequality as in Corollary \ref{th8.1}, which corresponds to the critical case ($a=-\frac{N}{p}+1$) of the inequality from \cite[Theorem 4.1]{RSY17a} when the Cauchy-Schwarz inequality is applied to the right-hand side.
\end{rem}
\begin{cor} \label{th8.1}
Let $\mathbb{G}$ be a stratified group with N being the dimension of the first stratum, $1<p,q<\infty$, $0<r<\infty$, with $p+q\geq r$, $\delta \in [0, 1] \cap \left[\frac{r-q}{r},\frac{p}{r}\right]$ and $b,c \in \mathbb{R}$. Assume that 
$$\frac{\delta r}{p}+\frac{(1-\delta)r}{q}=1 \ \ and  \ \ c=-\frac{N}{p}\delta+b(1-\delta).$$
Then we have the following critical Caffarelli-Kohn-Nirenberg type inequality for all $f\in C^\infty_0(\mathbb{G} \backslash \{{x'=0}\})$:
    
\begin{equation} \label{c.caf on strtf}
\||x'|^c f\|_{L^r(\mathbb{G})}\leq p^\delta \||x'|^{-\frac{N}{p}} \log|x'| (x'\cdot\nabla_H)f \|^\delta_{L^p(\mathbb{G})}\||x'|^bf\|^{1-\delta}_{L^q(\mathbb{G})},
\end{equation} where $|x'|$  is the Euclidean norm on $\mathbb{R}^N$, $\nabla_H$ is the horizontal gradient on $\mathbb{G}$. The constant in the inequality (\ref{c.caf on strtf}) is sharp for $p=q$ with $\frac{N}{p}=-b$ or for $\delta=0,1.$
\end{cor}

\begin{cor}[Uncertainty type principle with a logarithmic weight] \label{uncertainty on strtfd}
    Let $\mathbb{G}$ be a stratified group with N being the dimension of the first stratum.  Then for any $f \in C^\infty_0(\mathbb{G}\backslash\{{x'=0}\})$ and $\frac{1}{N}+\frac{1}{q}=1,$ we have
    \begin{equation} \label{crit. uncertainty on strtfd} 
     \int_{\mathbb{G}}|f|^2dx\leq N\left\||x'|^{-1}\log|x'|(x'\cdot\nabla_H) f\right\|_{L^N(\mathbb{G})} \left\||x'|f\right\|_{L^q(\mathbb{G})},
    \end{equation}
where $|x'|$ is the Euclidean norm on $\mathbb{R}^N$, $\nabla_H$ is horizontal gradient on $\mathbb{G}$.
\end{cor}

By using inequality (\ref{high.or ineq. on strtfd}) we obtain the higher order Caffarelli-Kohn-Nirenberg type inequality with a logarithmic weight on stratified Lie groups.

\begin{cor}
Let $\mathbb{G}$ be a stratified group with N being the dimension of the first stratum. Let $1<q<\infty$, $0<r<\infty$, with $q+2\geq r$, $\delta \in [0, 1] \cap \left[\frac{r-q}{r},\frac{2}{r}\right]$ and $b,c \in \mathbb{R}$. Assume that 
$$\frac{\delta r}{2}+\frac{(1-\delta)r}{q}=1 \ \ and  \ \ c=-\frac{N}{2}\delta+b(1-\delta).$$
Then we have the following higher order Caffarelli-Kohn-Nirenberg type inequality with a logarithmic weight for all $f\in C^\infty_0(\mathbb{G} \backslash \{{x'=0}\})$ and for any $k\in\mathbb{N}$:
\begin{equation} \label{h.c.caf on strtfd}
\||x'|^c f\|_{L^r(\mathbb{G})}\leq \bigg(\frac{2^k}{a(k)}\bigg)^\delta \||x'|^\frac{-N}{2}(\log|x'|)^k(x'\cdot \nabla_H)^k f\|_{L^2(\mathbb{G})}^\delta \||x'|^bf\|^{1-\delta}_{L^q(\mathbb{G})},
\end{equation} 
where $|x'|$ is the Euclidean norm on $\mathbb{R}^N$, $\nabla_H$ is the horizontal gradient on $\mathbb{G}$ and $a(k)=(2k-1)!!$ is the double factorial of an odd number.
\end{cor}

Now, let us present the aforementioned results on homogeneous Lie groups.
\vspace{4pt}

\begin{thm} \label{Lp caffarelli with remainder on hom}
Let $\mathbb{G}$ be a homogeneous group of homogeneous dimension $Q$, and let $|\cdot|$ be any homogeneous quasi-norm on $\mathbb{G}$. Let $1<p,q<\infty$, $0<r<\infty$ with $p+q\geq r$, $\delta \in [0, 1] \cap \left[\frac{r-q}{r},\frac{p}{r}\right]$ and $b,c \in \mathbb{R}$. Assume that 
$$\frac{\delta r}{p}+\frac{(1-\delta)r}{q}=1 \ \ and  \ \ c=-\frac{Q}{p}\delta+b(1-\delta).$$
Then we have the following Caffarelli-Kohn-Nirenberg type inequality with a logarithmic weight for any complex-valued function $f\in C^\infty_0(\mathbb{G} \backslash \{{0}\})$:
\small    
\begin{multline} \label{Lp c.caf with remainder on hom}
\||x|^c f\|_{L^r(\mathbb{G})}\leq p^\delta \Bigg(\Bigg\|\frac{\log|x|\mathcal{R}_{|x|}f}{|x|^{\frac{Q}{p}-1}}\Bigg\|_{L^p(\mathbb{G})}^p\\-\frac{1}{p^p}\int_\mathbb{G}C_p\left(\frac{p\log|x|\mathcal{R}_{|x|}f}{|x|^{\frac{Q}{p}-1}},\frac{p(\log|x|)^{-\frac{1}{p}+1}}{|x|^{\frac{Q}{p}-1}}\mathcal{R}_{|x|}\left(f(\log|x|)^\frac{1}{p}\right)\right)dx\Bigg)^\frac{\delta}{p}
    \||x|^bf\|^{1-\delta}_{L^q(\mathbb{G})},
\end{multline} 
\normalsize
where $|x|$  is any homogeneous quasi-norm on $\mathbb{G}$, $\mathcal{R}_{|x|}$ is the radial derivative on $\mathbb{G}$. The constant in the inequality \eqref{Lp c.caf with remainder on hom} is sharp for $p=q$ with $\frac{Q}{p}=-b$ or for $\delta=0,1.$  
\end{thm}
\begin{rem} Dropping the remainder term in \eqref{Lp c.caf with remainder on hom} implies results \cite[Theorem 7.1]{RSY18a} and \cite[Theorem 2.1]{RSY17b}.  
\end{rem}

\begin{proof}[Proof of Theorem \ref{Lp caffarelli with remainder on hom}.]
Case $\delta=0$. In this case, we have $q=r$ and $b=c$ by $\frac{\delta r}{2}+\frac{(1-\delta)r}{q}=1$ and $c=-\frac{Q}{2}\delta+b(1-\delta),$ respectively. Then, the inequality (\ref{Lp c.caf with remainder on hom}) reduces to the trivial estimate
$$\||x|^b f\|_{L^q(\mathbb{G})}\leq \|x|^bf\|_{L^q(\mathbb{G})}.$$
Case $\delta=1$. Notice that in this case, $r=p$ and $c=-\frac{Q}{p}$. By \eqref{rem. of crit. sob. on hom}, we have the equality
\begin{multline}
\left\|\frac{f}{|x|^\frac{Q}{p}}\right\|_{L^p(\mathbb{G})}= p\Bigg(\left\|\frac{\log|x|\mathcal{R}_{|x|} f}{|x|^{\frac{Q}{p}-1}} \right\|_{L^p(\mathbb{G})}^p\\-\frac{1}{p^p}\int_{\mathbb{G}}C_p\left(\frac{p\log|x|\mathcal{R}_{|x|}f}{|x|^{\frac{Q}{p}-1}},\frac{p(\log|x|)^{-\frac{1}{p}+1}}{|x|^{\frac{Q}{p}-1}}\mathcal{R}_{|x|}\left(f(\log|x|)^\frac{1}{p}\right)\right)dx\Bigg)^\frac{1}{p}.
\end{multline}
Case $\delta\in(0,1)\cap \left [\frac{r-q}{r},\frac{p}{r}\right] $. Taking into account $c=-\frac{Q}{p}\delta+b(1-\delta)$, a direct calculation gives 
$$\||x|^c f\|_{L^r(\mathbb{G})} = \left( \int_{\mathbb{G}}|x|^{cr}|f|^rdx \right)^\frac{1}{r} = \left( \int_{\mathbb{G}}\frac{|f|^{\delta r}}{|x|^{\frac{Q}{p}\delta r}}\frac{|f|^{(1-\delta)r}}{|x'|^{-b(1-\delta)r}}dx \right)^\frac{1}{r}.$$
Since we have $\delta\in(0,1)\cap \left [\frac{r-q}{r},\frac{p}{r}\right]$ and $q+p\geq r$, then by using H\"older's inequality for $\frac{\delta r}{p}+\frac{(1-\delta)r}{q}=1$, we obtain 
\begin{equation} \label{holder.r,p,q on hom}
\begin{split}
\||x|^c f\|_{L^r(\mathbb{G})}&\leq \left(\int_{\mathbb{G}}\frac{|f|^p}{|x|^Q}dx\right)^\frac{\delta}{p}\left(\int_{\mathbb{G}}\frac{|f|^q}{|x|^{-bq}}dx\right)^\frac{1-\delta}{q} \\ &= \Bigg \|\frac{f}{|x|^\frac{Q}{p}}\Bigg \|^\delta_{L^p}\Bigg \|\frac{f}{|x|^{-b}}\Bigg \|^{1-\delta}_{L^q}.
\end{split}
\end{equation}

By (\ref{rem. of crit. sob. on hom}) we have the identity
\begin{multline} \label{rem. of crit. sob. ^delta on hom}
\left\|\frac{f}{|x|^\frac{Q}{p}}\right\|_{L^p(\mathbb{G})}^\delta = p^\delta \Bigg(\left\|\frac{\log|x|\mathcal{R}_{|x|}f}{|x|^{\frac{Q}{p}-1}}\right\|_{L^p(\mathbb{G})}^p\\-\frac{1}{p^p}\int_{\mathbb{G}}C_p\left(\frac{p\log|x|\mathcal{R}_{|x|}f}{|x|^{\frac{Q}{p}-1}},\frac{p(\log|x|)^{-\frac{1}{p}+1}}{|x|^{\frac{Q}{p}-1}}\mathcal{R}_{|x|}\left(f(\log|x|)^\frac{1}{p}\right)\right)dx\Bigg)^\frac{\delta}{p}.
\end{multline}
Putting this in \eqref{holder.r,p,q on hom}, one has
\small
\begin{multline}
\||x|^c f\|_{L^r(\mathbb{G})}\leq p^\delta \Bigg(\left\|\frac{\log|x|\mathcal{R}_{|x|}f}{|x|^{\frac{Q}{p}-1}}\right\|_{L^p(\mathbb{G})}^p\\-\frac{1}{p^p}\int_{\mathbb{G}}C_p\left(\frac{p\log|x|\mathcal{R}_{|x|}f}{|x|^{\frac{Q}{p}-1}},\frac{p(\log|x|)^{-\frac{1}{p}+1}}{|x|^{\frac{Q}{p}-1}}\mathcal{R}_{|x|}\left(f(\log|x|)^\frac{1}{p}\right)\right)dx\Bigg)^\frac{\delta}{p}\Bigg \|\frac{f}{|x|^{-b}}\Bigg \|^{1-\delta}_{L^q}.
\end{multline}
\normalsize
In the case $q=p$ and $\frac{Q}{p}=-b$, H\"older's equality condition for the inequality \eqref{holder.r,p,q on hom} holds true for any function. Therefore, the constant $p^{\delta}$ in \eqref{Lp c.caf with remainder on hom} is sharp when $q=p$ and $\frac{Q}{p}=-b$.
\end{proof}

\begin{cor}[Uncertainty type principle with a logarithmic weight] \label{uncertainty on hom}
    Let $\mathbb{G}$ be a homogeneous group of homogeneous dimension $Q$, and let $|\cdot|$ be any homogeneous quasi-norm on $\mathbb{G}$. Let $\mathcal{R}_{|x|}=\frac{d}{d|x|}$ be the radial derivative. Then for any $f \in C^\infty_0(\mathbb{G}\backslash\{{0}\})$ and $\frac{1}{Q}+\frac{1}{q}=1,$ we have
    \begin{equation} \label{crit. uncertainty on hom} 
     \int_{\mathbb{G}}|f|^2dx\leq Q\left\|\log|x|\mathcal{R}_{|x|} f\right\|_{L^Q(\mathbb{G})} \left\||x|f\right\|_{L^q(\mathbb{G})}.
    \end{equation}

\end{cor}
\begin{proof}[Proof of Corollary \ref{uncertainty on hom}]
Dropping the remainder term in \eqref{Lp c.caf with remainder on hom} and taking $r=2$, $c=0$, $p=Q$, $b=1$, and $\delta=\frac{1}{2}$ implies   
    \begin{equation*}   
    \|f\|_{L^2(\mathbb{G})}\leq Q^\frac{1}{2} \Bigg\|\log|x|\mathcal{R}_{|x|}f\Bigg\|^{\frac{1}{2}}_{L^Q(\mathbb{G})}\||x|f\|^\frac{1}{2}_{L^q(\mathbb{G})},
    \end{equation*}
which implies \eqref{crit. uncertainty on hom}.
\end{proof}
\begin{rem}
In case $Q=q=2$, the inequality \eqref{crit. uncertainty on hom} reduces to 
\begin{equation} \label{critical case of HPW on hom}     \left(\int_{\mathbb{G}}|f|^2dx\right)^2 \leq 4 \int_{\mathbb{G}} (\log|x|)^2\left|\mathcal{R}_{|x|} f\right|^2dx\int_{\mathbb{G}}|x|^2 |f|^2dx,
\end{equation}
where $|x|$ is any homogeneous quasi-norm on $\mathbb{G}$ and $\mathcal{R}_{|x|}$ is the radial derivative on $\mathbb{G}$. The inequality \eqref{critical case of HPW on hom} is the critical case of the Heisenberg-Pauli-Weyl uncertainty type principle from \cite[Proposition 2.9]{RS17a}.
\end{rem}

By substituting inequality (\ref{high.or ineq. on hom}) into $p=2$ version of \eqref{holder.r,p,q on hom} we obtain the higher order Caffarelli-Kohn-Nirenberg type inequality with a logarithmic weight on homogeneous Lie groups.

\begin{cor}
Let $\mathbb{G}$ be a homogeneous group of homogeneous dimension $Q$, and let $|\cdot|$ be any homogeneous quasi-norm on $\mathbb{G}$. Let $1<q<\infty$, $0<r<\infty$, with $q+2\geq r$, $\delta \in [0, 1] \cap \left[\frac{r-q}{r},\frac{2}{r}\right]$ and $b,c \in \mathbb{R}$. Assume that 
$$\frac{\delta r}{2}+\frac{(1-\delta)r}{q}=1 \ \ and  \ \ c=-\frac{Q}{2}\delta+b(1-\delta).$$
Then we have the following higher order critical Caffarelli-Kohn-Nirenberg type inequality for all $f\in C^\infty_0(\mathbb{G} \backslash \{{0}\})$:
\begin{equation} \label{h.c.caf on hom}
\||x|^c f\|_{L^r(\mathbb{G})}\leq \bigg(\frac{2^k}{a(k)}\bigg)^\delta \||x|^{-\frac{Q}{2}+1}(\log|x|)^k(\mathcal{R}_{|x|})^k f\|_{L^2(\mathbb{G})}^\delta \||x|^bf\|^{1-\delta}_{L^q(\mathbb{G})},
\end{equation} 
where $\mathcal{R}_{|x|}=\frac{d}{d|x|}$ is the radial derivative and $a(k)=(2k-1)!!$ is the double factorial of an odd number.
\end{cor}

\section{Appendix}
In this section, we give the proof of Lemma \ref{lemma 1.} and Lemma \ref{lemma 2}.
\begin{proof} [Proof of Lemma \ref{lemma 1.}] We will prove this lemma by induction. \\

Case $\kappa=1.$ It can be easily verified that (\ref{lemma 1}) holds for  $\kappa=1$ as follows:
\begin{multline*}
    \log|x'|(x'\cdot \nabla_N) \big(\log|x'|(x'\cdot \nabla_N) f\big)\\=\log|x'|(x'\cdot \nabla_N\log|x'|)(x'\cdot\nabla_N)f+(\log|x'|^2(x'\cdot\nabla_N)^2f\\=\log|x'|\sum_{j=1}^N x'_j\frac{\frac{x'_j}{|x'|}}{|x'|}(x'\cdot\nabla_N)f+(\log|x'|)^2(x'\cdot\nabla_N)^2f\\=\log|x'|(x'\cdot \nabla_N) f+(\log|x'|)^2(x'\cdot \nabla_N)^2f.
\end{multline*}
Now assuming that (\ref{lemma 1}) holds for $k^{th}$ order, let us show that it also holds for $(k+1)^{th}$ order.\\

Case $\kappa=k+1$. Applying $\log|x'|(x'\cdot\nabla_N)$ to both sides of the $k^{th}$ order of (\ref{lemma 1}) we get
\begin{multline*}
\log|x'|(x'\cdot\nabla_N)\big[(\log|x'|)^k(x'\cdot \nabla_N)^k\big(\log|x'|(x'\cdot \nabla_N\big) f)\big] 
\\=\log|x'|(x'\cdot\nabla_N)\big[k(\log|x'|)^k(x'\cdot \nabla_N)^k f+(\log|x'|)^{k+1}(x'\cdot \nabla_N)^{k+1}f\big].
\end{multline*}    
By simplifying the above equality, we get
\begin{multline} \label{k+1th}
k(\log|x'|)^k(x'\cdot \nabla_N)^k(\log|x'|(x'\cdot \nabla_N)f)+(\log|x'|)^{k+1}(x'\cdot \nabla_N)^{k+1}(\log|x'|(x'\cdot \nabla_N)f)\\=k^2(\log|x'|)^k(x'\cdot \nabla_N)^kf+k(\log|x'|)^{k+1}(x'\cdot \nabla_N)^{k+1}f\\+(k+1)(\log|x'|)^{k+1}(x'\cdot \nabla_N)^{k+1}f+(\log|x'|)^{k+2}(x'\cdot \nabla_N)^{k+2}f.    
\end{multline}
Due to \eqref{lemma 1}: 
\begin{multline} \label{k*kth}
k(\log|x'|)^k(x'\cdot \nabla_N)^k(\log|x'|(x'\cdot \nabla_N)f)\\=k^2(\log|x'|)^k(x'\cdot \nabla_N)^kf+k(\log|x'|)^{k+1}(x'\cdot \nabla_N)^{k+1}f.
\end{multline}
Simplifying \eqref{k*kth} from \eqref{k+1th}  we obtain the $(k+1)^{th}$ order of (\ref{lemma 1}), which is 
\begin{multline*}
(\log|x'|)^{k+1}(x'\cdot \nabla_N)^{k+1}(\log|x'|(x'\cdot \nabla_N)f)\\=(k+1)(\log|x'|)^{k+1}(x'\cdot \nabla_N)^{k+1}f+(\log|x'|)^{k+2}(x'\cdot \nabla_N)^{k+2}f,    
\end{multline*}
completing the proof.
\end{proof}

\begin{proof} [Proof of Lemma \ref{lemma 2}]  
To prove \eqref{recurrence} we will use formula \eqref{coefficient}. As the main formula from \eqref{coefficient} is not applicable for $O(k,1)$ and $O(k,k)$ which are involved in \eqref{recurrence} when $m=2$ and $m=k$ respectively, let us prove \eqref{recurrence} for these cases first.\\

Case $m=2.$ In this case, $O(k,1)=a_k$ from \eqref{coefficient}. So, we need to prove that
\begin{equation} \label{m=2}
4k(k+1)O(k,2)+4a_k=O(k+1,2).    
\end{equation}
For this, by using \eqref{coefficient} and $k(k+1)=o_k$  we write
\begin{multline} \label{k-2}
    4k(k+1)O(k,2)+4a_k\\=\sum_{t=1}^{k-2}4^{k-t+1}a_t\sum_{\underbrace{t+1 \leq i_1<...<i_r\leq k-1}_{k-t-1}}\prod_{j=1}^{k-t-1}o_{i_j}o_k+4^{2}a_{k-1}o_k+4a_k\\=\sum_{t=1}^{k-2}4^{k-t+1}a_t \overbrace{(o_{t+1}o_{t+2}...o_{k-1}o_k)}^{k-t}
    +4^{2}a_{k-1}o_k+4a_k,
\end{multline}
where we have used that
$$\sum_{\underbrace{t+1 \leq i_1<...<i_r\leq k-1}_{k-t-1}}\prod_{j=1}^{k-t-1}o_{i_j}=o_{t+1}o_{t+2}...o_{k-1}$$
since the range of index $j$ and the range of index $i_j$ have the same length.\\
Now using 
$$\sum_{t=1}^{k-2}4^{k-t+1}a_t\sum_{t+1 \leq i_1<...<i_r\leq k} \prod_{j=1}^{k-t}o_{i_j}+4^{2}a_{k-1}o_k=\sum_{t=1}^{k-1}4^{k-t+1}a_t\sum_{t+1 \leq i_1<...<i_r\leq k} \prod_{j=1}^{k-t}o_{i_j},$$
we derive from \eqref{k-2} that
\begin{equation*}
    4k(k+1)O(k,2)+4a_k=\sum_{t=1}^{k-1}4^{k-t+1}a_t\sum_{t+1 \leq i_1<...<i_r\leq k} \prod_{j=1}^{k-t}o_{i_j}+4a_k=O(k+1,2)
\end{equation*}
due to the definition \eqref{coefficient}. Thus, we have proved \eqref{m=2}.\\

Case $m=k.$ In this case, $O(k,k)=4^{k-1}$ from \eqref{coefficient}. Thus we need to prove that
\begin{equation} \label{m=k}
4k(k+1)4^{k-1}+4O(k,k-1)=O(k+1,k).
\end{equation}
By using \eqref{coefficient} and $k(k+1)=o_k$ on the left-hand side of \eqref{m=k}  we have
\begin{multline*}
    4k(k+1)4^{k-1}+4O(k,k-1)\\=
    4k(k+1)4^{k-1}+\sum_{t=1}^{1}4^{k-t+1}a_t\sum_{t+1 \leq i_1<...<i_r\leq k-1} \prod_{j=1}^{2-t}o_{i_j}+4^{k-1}a_{2}\\=4^{k}o_k+4^{k}a_1(o_2+o_3...+o_{k-1})+4^{k-1}a_{2}.
\end{multline*}
As $a_1=1$, we can rewrite the above formula and get
\begin{multline*}
    4k(k+1)4^{k-1}+4O(k,k-1)=4^{k}a_1(o_2+o_3...+o_{k-1}+o_k)+4^{k-1}a_{2}\\=\sum_{t=1}^{1}4^{k-t+1}a_t\sum_{t+1 \leq i_1<...<i_r\leq k} \prod_{j=1}^{2-t}o_{i_j}+4^{k-1}a_{2}=O(k+1,k)
\end{multline*}
due to the definition \eqref{coefficient}. Thus, we have proved \eqref{m=k}.\\

Now we can prove \eqref{recurrence} for $2<m<k.$ By using the formula from \eqref{coefficient} and $o_k=k(k+1)$ on the left-hand side of \eqref{recurrence} we get 

\begin{multline} \label{rec1}
    4k(k+1)O(k,m)+4O(k,m-1)\\=\sum_{t=1}^{k-m}4^{k-t+1}a_t\sum_{t+1 \leq i_1<...<i_r\leq k-1} \prod_{j=1}^{r}o_{i_j}o_k+4^{m}a_{k-m+1}o_k\\+\sum_{t=1}^{k-m+1}4^{k-t+1}a_t\sum_{t+1 \leq i_1<...<i_{r+1}\leq k-1} \prod_{j=1}^{r+1}o_{i_j}+4^{m-1}a_{k-m+2},
\end{multline} 
where $r=k-m-t+1$. \\
By adding summations in \eqref{rec1} for $t=1,...,k-m$ and writing the term of the second summation for $t=k-m+1$ separately we get
\begin{multline} \label{rec2}
    4k(k+1)O(k,m)+4O(k,m-1)\\=\sum_{t=1}^{k-m}4^{k-t+1}a_t\bigg(\underbrace{\sum_{t+1 \leq i_1<...<i_r\leq k-1} \prod_{j=1}^{r}o_{i_j}o_k}_{\sigma_1}+\underbrace{\sum_{t+1 \leq i_1<...<i_{r+1}\leq k-1} \prod_{j=1}^{r+1}o_{i_j}}_{\sigma_2}\bigg)\\+4^{m}a_{k-m+1}o_k+4^{m}a_{k-m+1}\sum_{k-m+2 \leq i_1\leq k-1} \prod_{j=1}^{1}o_{i_j}+4^{m-1}a_{k-m+2}.
\end{multline}
By adding summations $\sigma_1$ and $\sigma_2$ in \eqref{rec2} we get sum of all possible products of $r+1$ Oblong numbers from $t+1$ to $k$, as follows: 
\begin{multline} \label{rec3}
\sigma_1+\sigma_2=\sum_{t+1 \leq i_1<...<i_r\leq k-1} \overbrace{o_{i_1}o_{i_2}...o_{i_r}o_k}^{r+1}+\sum_{t+1 \leq i_1<...<i_{r+1}\leq k-1} \overbrace{o_{i_1}o_{i_2}o_{i_3}...o_{i_{r+1}}}^{r+1}\\=\sum_{t+1 \leq i_1<...<i_{r+1}\leq k} \prod_{j=1}^{r+1}o_{i_j}.  
\end{multline}
Now, by combining \eqref{rec2} and \eqref{rec3} we get
\begin{multline} \label{}
    4k(k+1)O(k,m)+4O(k,m-1)\\=\sum_{t=1}^{k-m}4^{k-t+1}a_t\sum_{t+1 \leq i_1<...<i_{r+1}\leq k} \prod_{j=1}^{r+1}o_{i_j}\\+4^{m}a_{k-m+1}(o_{k-m+2}+...+o_{k-1}+o_k)+4^{m-1}a_{k-m+2}\\=\sum_{t=1}^{k-m}4^{k-t+1}a_t\sum_{t+1 \leq i_1<...<i_{r+1}\leq k} \prod_{j=1}^{r+1}o_{i_j}+4^{m}a_{k-m+1}\sum_{k-m+2 \leq i_1\leq k} \prod_{j=1}^{1}o_{i_j}+4^{m-1}a_{k-m+2}\\=\sum_{t=1}^{k-m+1}4^{k-t+1}a_t\sum_{t+1 \leq i_1<...<i_{r+1}\leq k} \prod_{j=1}^{k-m-t+2}o_{i_j}+4^{m-1}a_{k-m+2}=O(k+1,m),
\end{multline} 
which implies \eqref{recurrence} for $2<m<k$. Thus, we have proved \eqref{recurrence} for $2\leq m\leq k$.
\end{proof}


\begin{thebibliography}{9}

\bibitem[BEHL08]{BEHL08}
A.~Balinsky, W. D.~Evans, D. Hundertmark and R. T.~Lewis.
\newblock On inequalities of Hardy-Sobolev type.
\newblock {\em Banach J. Math. Anal.}, 2(2):94-106, 2008.

\bibitem[Ber00]{Ber00}
G.~Bertin.
\newblock Dynamics of Galaxies.
\newblock {\em Cambridge University Press, Cambridge}, 2000.

\bibitem[BF10]{BF10}
P.~Blasiak and Ph.~Flajolet.
\newblock Combinatorial models
of creation-annihilation.
\newblock {\em Semin. Lothar. Comb.}, 65, Art. No. B65c, 2010.

\bibitem[BLU07]{BLU07}
A.~Bonfiglioli, E.~Lanconelli and F.~Uguzzoni.
\newblock Stratified Lie Groups and Potential
Theory for Their Sub-Laplacians.
\newblock {\em Springer, Berlin}, 2007.

\bibitem[BPS03]{BPS03}
P.~Blasiak, K.A.~Penson and A.I.~Solomon.
\newblock The general boson normal ordering problem.
\newblock {\em Phys. Lett. A}, 309(3-4):198–205, 2003.




\bibitem[BT02]{BT02}
M.~Badiale and G.~Tarantello.
\newblock A Sobolev–Hardy inequality with applications to a nonlinear elliptic equation arising in astrophysics.
\newblock {\em Arch. Ration. Mech. Anal.}, 163:259–293, 2002.

\bibitem[CKLL24]{CKLL24}
C.~Cazacu, D.~Krejčiřík, N.~Lam and A.~Laptev. 
\newblock Hardy inequalities for magnetic p-Laplacians.
\newblock {\em Nonlinearity}, 37(3) 035004, 2024.

\bibitem[CH02]{CH02}
C. A. Charalambides. 
\newblock Enumerative combinatorics.
\newblock {\em CRC Press}, 2002.

\bibitem[CCR15]{CCR15}
P.~Ciatti, M.~Cowling and F.~Ricci,
\newblock Hardy and uncertainty inequalities on stratified Lie groups.
\newblock {\em Adv. Math.}, 277:365-387, 2015.

\bibitem[Cio01]{Cio01}
L.~Ciotti.
\newblock Dynamical models in astrophysics.
\newblock {\em Scuola Normale Superiore, Pisa}, 2001.

\bibitem[DN20]{DN20}
N.T.~Duy and H.B.~Nguyen.
\newblock Cylindrical {H}ardy inequalities on half-spaces.
\newblock {\em Electron. J. Differential Equations}, 2020: Art. No. 75, 2020.

\bibitem[DP21]{DP21}
N.T.~Duy and L.L.~Phi.
\newblock Cylindrical {H}ardy type inequalities with {B}essel pairs.
\newblock {\em Oper. Matrices}, 15(2):485--495, 2021.

\bibitem[EGH15]{EGH15}
J.~Engbers, D.~Galvin and J.~Hilyard.
\newblock Combinatorially interpreting generalized Stirling numbers.
\newblock {\em European J. Combin.}, 43:32-54, 2015.

\bibitem[Fol75]{FOL75}
G. B.~Folland.
\newblock Subelliptic estimates and function spaces on nilpotent Lie groups.
\newblock {\em Ark. Mat.}, 13(2):161–207, 1975.

\bibitem[FR16]{FR16}
V.~Fischer and M.~Ruzhansky.
\newblock Quantization on Nilpotent Lie Groups.
\newblock {\em Volume 314
of Progress in Mathematics, Birkh\"auser, Basel}, 2016. (open access book)

\bibitem[FS82]{FS82}
G. B. Folland and E. M. Stein.
\newblock Hardy spaces on homogeneous groups.
\newblock {\em Mathematical Notes, Princeton University Press}, 28, 1982.

\bibitem[IIO16]{IIO16}
N. Ioku, M. Ishiwata and T. Ozawa.
\newblock Sharp remainder of a critical Hardy inequality.
\newblock {\em Arch. Math. (Basel)}, 106(1):65–71, 2016.

\bibitem[IIO17]{IIO17}
N. Ioku, M. Ishiwata and T. Ozawa.
\newblock Hardy type inequalities in $L_{p}$ with sharp remainders.
\newblock {\em J. Inequal. Appl.}, 2017: Art. No. 5, 2017.

\bibitem[JM97]{JM97}
A.~Joarder and M.~Munir.
\newblock Classroom note: An inductive derivation of Stirling numbers of the second kind and their applications in statistics.
\newblock {\em  J. Appl. Math. Decis. Sci.}, 1(2):151-157, 1997.

\bibitem[KY24]{KY24}
M.~Kalaman and N.~Yessirkegenov.
\newblock Cylindrical Hardy, Sobolev type and Caffarelli-Kohn-Nirenberg type inequalities and identities.
\newblock {\em  arxiv preprint arXiv:2407.08393}, 2024.

\bibitem[L00]{L00}
W.~Lang.
\newblock On generalizations of the Stirling number triangles.
\newblock {\em J. Integer Seq.}, 3, Art. No. 00.2.4, 2000.

\bibitem[Maz11]{Maz11}
V.~Maz’ya.
\newblock Sobolev Spaces with Applications to Elliptic Partial Differential Equations. Second, revised
and augmented edition.
\newblock {\em Grundlehren der Mathematischen Wissenschaften [Fundamental Principles of
Mathematical Sciences]}, Springer, {V}ol. 342, 2011.

\bibitem[MBP05]{MBP05}
M. A.~Méndez, P.~Blasiak and K. A.~Penson.
\newblock Combinatorial approach to generalized Bell and Stirling numbers and boson normal ordering problem.
\newblock {\em J. Math. Phys.}, 46(8), Art. No.  083511, 2005.

\bibitem[M10]{M10}
M.~Mohammad-Noori.
\newblock Some remarks about the derivative operator and generalized Stirling numbers.
\newblock {\em Ars Comb.}, 100:177-192, 2010.


\bibitem[OS09]{OS09}
T.~Ozawa and H.~Sasaki.
\newblock Inequalities associated with dilations.
\newblock {\em Commun. Contemp. Math.}, 11(2):265–277, 2009.

\bibitem[RS17a]{RS17}
M.~Ruzhansky and D.~Suragan.
\newblock Layer potentials, Kac’s problem, and refined Hardy inequality on homogeneous Carnot groups.
\newblock {\em Adv. Math.}, 308:483-528, 2017.

\bibitem[RS17b]{RS17a}
M.~Ruzhansky and D.~Suragan.
\newblock Uncertainty relations on nilpotent Lie groups.
\newblock {\em Proc. Roy. Soc. Edinburgh Sect. A}, 473(2201), Art. No. 20170082, 2017.

\bibitem[RS17c]{RS17b}
M.~Ruzhansky and D.~Suragan.
\newblock On horizontal Hardy, Rellich, Caffarelli–Kohn–Nirenberg and p-sub-Laplacian inequalities on stratified groups.
\newblock {\em J. Differential Equations}, 262(3):1799-1821, 2017.

\bibitem[RS19]{RS19}
M.~Ruzhansky and D.~Suragan.
\newblock Hardy inequalities on homogeneous groups: 100 years of Hardy inequalities.
\newblock {\em Volume 327
of Progress in Mathematics, Birkh\"auser, Basel}, 2019. (open access book)

\bibitem[RSY17a]{RSY17a}
M.~Ruzhansky, D.~Suragan, and N.~Yessirkegenov.
\newblock Caffarelli-Kohn-Nirenberg and Sobolev type inequalities on stratified Lie groups.
\newblock {\em NoDEA Nonlinear Differential Equations Appl.}, 24, Art. No. 56, 2017.

\bibitem[RSY17b]{RSY17b}
M.~Ruzhansky, D.~Suragan, and N.~Yessirkegenov.
\newblock Extended Caffarelli--Kohn--Nirenberg inequalities and superweights for $L^p$-weighted Hardy inequalities.
\newblock {\em C. R. Math. Acad. Sci. Paris}, 355(6):694-698, 2017.

\bibitem[RSY18a]{RSY18a}
M.~Ruzhansky, D.~Suragan, and N.~Yessirkegenov.
\newblock Extended Caffarelli-Kohn-Nirenberg inequalities, and remainders, stability, and superweights for $L^p$-weighted Hardy inequalities.
\newblock {\em Trans. Amer. Math. Soc. Ser. B}, 5(2):32-62, 2018.

\bibitem[RSY18b]{RSY18b}
M.~Ruzhansky, D.~Suragan, and N.~Yessirkegenov.
\newblock Sobolev type inequalities, Euler-Hilbert-Sobolev and Sobolev-Lorentz-Zygmund spaces on homogeneous groups.
\newblock {\em Integral Equ. Oper. Theory}, 90, Art. No. 10, 2018.

\bibitem[RY24a]{RY24a}
M.~Ruzhansky and N.~Yessirkegenov.
\newblock Hardy–Sobolev–Rellich, Hardy–Littlewood–Sobolev and Caffarelli–Kohn–Nirenberg Inequalities on General Lie Groups.
\newblock {\em J. Geom. Anal.}, 34, Art. No. 223, 2024. 

\bibitem[RY24b]{RY24b}
M.~Ruzhansky and N.~Yessirkegenov.
\newblock Hypoelliptic functional inequalities.
\newblock {\em Math. Z.}, 307, Art. No. 22, 2024. 

\bibitem[San22]{San22}
M.~Sano.
\newblock Improvements and generalizations of two Hardy type inequalities and their applications to the Rellich type inequalities.
\newblock {\em Milan J. Math.}, 90:647--678, 2022.

\bibitem[Sch1823]{Sch1823}
H. F. Scherk.
\newblock De evolvenda functione $\frac{ydydyd...ydX}{dx^n}
$ disquisitiones nonnullaeanalyticae.
\newblock {\em Ph.D. thesis, Berlin}, 1823. Publicy available from
G\"ottinger Digitalisierungszentrum (GDZ)

\bibitem[SSW03]{SSW03}
S.~Secchi, D.~Smets and M.~Willen.
\newblock Remarks on a Hardy-Sobolev inequality.
\newblock {\em C.R.Acad.Sci.
Paris, Ser.I}, 336:811–815, 2003.

\end{thebibliography}
\end{document}